\documentclass{amsart}
\usepackage{tabularx,booktabs}
\usepackage{amsmath}
\usepackage{caption}
\usepackage{amsfonts}
\usepackage{amscd}
\usepackage{amssymb} \usepackage{latexsym}

\usepackage{euscript}
\usepackage{epsfig}
\usepackage{graphics}
\usepackage{array}
\usepackage{enumerate}
\usepackage{dsfont}
\usepackage{color}
\usepackage{wasysym}
\usepackage{hyperref}
\usepackage{pdfsync}
\def\sharp{\#}

\newcommand{\bel}[1]{\begin{equation}\label{#1}}

\newcommand{\be}{\begin{equation}}

\newcommand{\ba}{\begin{eqnarray}}
\newcommand{\ea}{\end{eqnarray}}

\newcommand{\qe}{\end{equation}}
\newcommand{\R}{{\mathbb R}}
\newcommand{\N}{{\mathbb N}}

\newcommand{\SP}{\mathbb{S}}

\newcommand{\Ric}{\mathrm{Ric}_F}
\newcommand{\fmp}{\mathcal{FC}_{+}}
\newcommand{\fm}{\mathcal{FC}_{+}^{\mathbb{S}^2}}

\newcommand{\CC}{\mathcal{C}_{+}}

\newcommand{\Hmm}[1]{\leavevmode{\marginpar{\tiny%
$\hbox to 0mm{\hspace*{-0.5mm}$\leftarrow$\hss}%
\vcenter{\vrule depth 0.1mm height 0.1mm width \the\marginparwidth}%
\hbox to
0mm{\hss$\rightarrow$\hspace*{-0.5mm}}$\\\relax\raggedright #1}}}

\newtheorem{theorem}{Theorem}[section]

\newtheorem{lemma}[theorem]{Lemma}

\newtheorem{prop}[theorem]{Proposition}

\newtheorem{fact}{Fact}

\usepackage{graphicx,url,color}

\def\G{\mathcal{G}}
\def\PCnC{\mathcal{C}_+}

\def\elongatedpyramid#1{\foreach \x in {1,...,#1}{\draw (90+\x*360/#1:0)
-- (90+\x*360/#1:0.5) -- (90+\x*360/#1:1); \draw
(90+\x*360/#1:0.5)--(90+360/#1+\x*360/#1:0.5);
\draw (90+\x*360/#1:1)--(90+360/#1+\x*360/#1:1);
}}

\def\prism#1#2{\foreach \x in {1,...,#1}{\draw  (90+\x*360/#1:#2) -- (90+\x*360/#1:1); \draw
(90+\x*360/#1:#2)--(90+360/#1+\x*360/#1:#2);
\draw (90+\x*360/#1:1)--(90+360/#1+\x*360/#1:1);
}}

\def\antiprism#1#2{\foreach \x in {1,...,#1}{\draw[rotate=\x*360/#1]
(360/#1:#2) -- (180/#1:1) -- (0:#2)--cycle;
\draw[rotate=\x*360/#1](180/#1:1)--(-180/#1:1);
}}

\def\trapezocupola#1#2{\foreach \x in {1,...,#1}{\draw[rotate=360*\x/#1] (90+180/#1:#2)--(90:1)--(90+360/#1:1)--(90+180/#1:#2)--(0:0);}}

\usepackage{tikz}
\usetikzlibrary{calc,fadings,decorations.pathreplacing}
\usetikzlibrary{arrows.meta}

\begin{document}
\title[Graphs on surfaces with positive Frrman or corner cuvatures]{Graphs
on surfaces with positive Forman curvature or corner curvature}
\date{}

\author{Yohji Akama} \address[Yohji Akama]{Mathematical Institute, Graduate School of Science, Tohoku University,
Sendai, 980-0845, Japan.  \\ Tel.: +81-22-795-6402, Fax: +81-22-795-6400}
\email{yoji.akama.e8@tohoku.ac.jp}
\author{Bobo Hua}
 \address[Bobo Hua]{School of Mathematical Sciences, LMNS,
 Fudan University, Shanghai 200433, China; Shanghai Center for
Mathematical Sciences, 2005 Songhu Road, Shanghai 200438, China} \email{bobohua@fudan.edu.cn}
\author{
Yanhui Su}
\address[Yanhui Su]{
 College of Mathematics and Computer Science, Fuzhou University, Fuzhou
 350116, China} \email{suyh@fzu.edu.cn}
 \author{
 Haohang Zhang}
 \address[Haohang Zhang]{Shanghai Center for
Mathematical Sciences, Fudan University, Shanghai 200433,
China}
\email{18300180032@fudan.edu.cn}

\begin{abstract} 
 On one hand, we study the class of graphs on surfaces, satisfying
 tessellation properties, with positive Forman curvature on each edge. 
Via medial graphs, we provide a new proof for the finiteness of the
 class, and give a complete classification. On the other hand, we
 classify the class of graphs on surfaces with positive corner
 curvature.
\end{abstract}
\maketitle

\tableofcontents

\maketitle









\section{Introduction}
The Gaussian curvature of a smooth surface is well studied in differential geometry, which describes the convexity of the surface. For a polyhedron in $\R^3,$ the discrete Gaussian curvature, as a measure, concentrates on the set of vertices.  In the graph theory, the combinatorial curvature of a planar graph, which serves as the discrete Gaussian curvature of a canonical piecewise flat surface, was introduced by \cite{MR0279280,MR0410602,MR919829,Ishida90} respectively. It has been extensively studied in the literature; see e.g.
\cite{Woess98,MR1600371,MR1864922,MR1797301,MR1894115,MR1923955,MR2038013,SY04,RBK05,MR2243299,DM07,CC08,MR2466966,MR2470818,MR2558886,MR2818734,MR2826967,MR3624614,Gh17}.

Let $S$ be a (possibly noncompact) connected surface without boundary. Let $(V,E)$ be a (possibly infinite) locally finite, undirected, simple graph with the set of vertices $V$ and the set of edges $E.$ We may regard $(V,E)$ as a $1$-dimensional topological space. Let $\varphi:(V,E)\to S$ be an topological embedding. We denote by $F$ the set of faces induced by the embedding $\varphi$, i.e. connected components of the complement of the embedding image of $(V,E)$ in $S.$ We write $G=(V,E,F)$ for the cellular complex structure induced by the embedding, which is called a graph on a surface \cite{MT01} (or a semiplanar graph \cite{MR3318509}). For any $\sigma\in F,$ we denote by $\overline{\sigma}$ the closure of $\sigma$ in $S,$ which is called the closed face. We say that a semiplanar graph $G=(V,E,F)$ is a \emph{tessellation} of $S$ if the following hold, see e.g. \cite{MR2826967}:
\begin{enumerate}[(i)]
\item For any compact set $K\subset S,$ $K$ can be covered by finitely many closed faces.
\item Every closed face is homeomorphic to a closed disk whose boundary consists of finitely many edges of the graph.
\item Every edge is contained in exactly two different closed faces.
\item If two closed faces intersect, then the intersection is either a vertex or the closure of an edge.
\end{enumerate} 


One can show that a tessellation of $S$ is finite if and only if $S$ is compact. It is called a planar tessellation (resp. a tessellation in the real projective plane) if $S$ is the sphere $\SP^2$ or the plane $\R^2$ (resp. $\R P^2$). In this paper, we always consider tessellations, and call planar tessellations planar graphs for simplicity. 

For a graph on a surface $G=(V,E,F),$ two vertices $x,y$ are called neighbors if there is an edge connecting $x$ and $y,$ denoted by $x\sim y.$ Two elements in $V,E,F$ are called \emph{incident} if the closures of their embedding images intersect. If one is contained in the closure of the other, then we write like $x\prec e, e\prec \sigma, x\prec \sigma,$ for $x\in V, e\in E, \sigma\in F.$ We denote by $|x|$ (resp. $|\sigma|$) the degree of a vertex $x$ (resp. a face $\sigma$), i.e. the number of neighbors of $x$ (resp. the number of edges incident to $\sigma$). In this paper, we only consider semiplanar graphs satisfying the following: for any vertex $x$ and face $\sigma,$ $$|x|\geq 3,\ |\sigma|\geq 3.$$

Given a graph $G=(V,E,F)$ on a surface $S,$ 
the \emph{combinatorial curvature} at a vertex $x$ is defined as 
\begin{equation*}\label{def:comb}\Phi(x)=1-\frac{|x|}{2}+\sum_{\sigma\in
F: x\prec{\sigma}}\frac{1}{|\sigma|}.\end{equation*} 
We endow $S$ with a canonical piecewise flat metric as follows: assign each edge length one, replace each face by a regular Euclidean polygon of side-length one with same facial degree, and glue these polygons along the common edges; see \cite{BBI01} for gluing metrics. The ambient space $S$ equipped with the gluing metric is called the \emph{(regular Euclidean) polyhedral surface} of $G,$ denoted by $S(G).$ For the metric surface $S(G)$, the generalized Gaussian curvature is a measure concentrated on vertices, whose mass at each vertex $x$ is given by the angle defect $K(x),$ i.e. $2\pi$ minus the total angle at $x.$ One easily sees that
$$\Phi(x)=\frac{1}{2\pi}K(x),\quad \forall x\in V,$$ where $K(x)$ is  the angle defect, i.e. the discrete Gaussian curvature, at the vertex $x.$ If $S$ is compact, the discrete Gauss-Bonnet theorem reads as
\begin{equation}\label{eq:GBthm}\sum_{x\in V}\Phi(x)=\chi(S),\end{equation} where $\chi(\cdot)$ is the Euler characteristic of $S.$ 

A planar graph $G$ has nonnegative combinatorial curvature if and only if the polyhedral surface $S(G)$ is a generalized convex surface in the sense of Alexandrov, see \cite{BuragoGromovPerelman92,BBI01,MR3318509}. Higuchi \cite{MR1864922} conjectured a discrete Bonnet-Myers theorem that a graph on a surface $G$ with positive combinatorial curvature everywhere is finite. It was confirmed by DeVos and Mohar 
\cite{DM07}, see e.g. \cite{CC08,MR2470818,MR3624614,oldridge17:_charac,Gh17} for further developments.
\begin{theorem}[\cite{DM07}]\label{thm:DeVosMohar2} For a graph $G=(V,E,F)$ on a surface $S$, if 
$\Phi(x)>0$ for any $x\in V,$ then $G$ is finite. Moreover, $S=\SP^2$ or $\R P^2.$
\end{theorem}

For a planar graph $G$ with an embedding, one can define the dual graph $G^*$: the vertices of $G^*$ are corresponding to faces of $G,$ the faces of $G^*$ are corresponding to vertices of $G,$ and two vertices in $G^*$ are adjacent if and only if the corresponding faces in $G$ share a common edge. One can show that for a finite graph $G,$ it is a tessellation if and only if so is $G^*,$ see e.g. \cite{MR2243299}.
For a graph on a surface $G$ with positive combinatorial curvature, the
dual graph may not have nonnegative combinatorial curvature; 7-gonal
prism~(resp. 5-gonal antiprism~(Figure~\ref{fig:ap6}~(right))) has positive combinatorial
curvature everywhere, but the dual, the 7-gonal bipyramid~(resp. 5-gonal
pseudo-double wheel~(Figure~\ref{p_expansion}~(upper right))), has negative
curvature at the two apexes. This means that the dual operation doesn't preserve the class of graphs on surfaces with positive combinatorial curvature. We will study some curvature notions of graphs on surfaces, for which the dual operation is closed on the class of graphs with positive curvature.

Motivated by Bochner techniques for differential forms in Riemannian geometry, Forman~\cite{MR1961004} introduced the curvature on general CW complexes, which is now called the Forman curvature. For a compact Riemannian manifold $M,$ let $\Delta_p$ be the Hodge Laplacian on $p$-forms, $p\in \N_0.$ The Bochner-Weitzenb\"ock formula reads as $$\Delta_p=(\nabla_p)^*\nabla_p+F_p,$$ where $\nabla_p$ is the Levi-Civita covariant derivative operator on $p$-forms and $F_p$ denotes the curvature operator on $p$-forms.
On a CW complex, Forman derived an analogous formula for the discrete Hodge Laplacian and defined the remainder term $F_p$ as the discrete curvature on $p$-cells. In this paper, we only consider the Forman curvature $F_1$ on $1$-cells, i.e. edges, with weight one everywhere. This curvature $F_1$ corresponds to the discrete analog of Ricci curvature, and we will denote it by $\Ric$ in this paper.
For a graph on a surface $G=(V,E,F),$
two edges $e_1$ and $e_2$ are called \emph{parallel neighbors} if one and only one of the following holds:
\begin{enumerate}\item There exists $x\in V$ such that $x\prec e_1,x\prec e_2.$ \item There 
exists $\sigma\in F$ such that $e_1\prec \sigma, e_2\prec \sigma.$
\end{enumerate} 
 The Forman curvature of an edge $e$ is defined as, see~\cite{MR1961004},
\begin{equation}\label{eq:deffm}\Ric(e)=\sharp \{\sigma\in F: e\prec \sigma\}+\sharp \{x\in V: x\prec e\}-\sharp\{\mathrm{parallel\ neighbors\ of }\ e\}.\end{equation} We are interested in 
graphs on surfaces with positive Forman curvature (everywhere). Note
 that a graph on surface with positive Forman curvature may not be a
 tessellation, see e.g. Figure~\ref{fig:columnC4}.
  \begin{figure}[ht]\centering
  \begin{tikzpicture}
   \draw (1,-1)-- (1,1) -- (-1,1)--(-1,-1)--cycle;
  \end{tikzpicture}
 \caption{The graph $C_4$ embedded in $\SP^2$. Each edge $e$ has $\Ric(e)=3.$\label{fig:columnC4}}
  \end{figure}
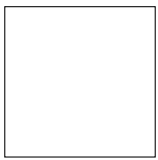
  There are infinitely many nontessellation, planar graphs with
 positive Forman curvature (everywhere). For example, the star graph
 $K_{1,n}$ $(n\ge1)$ has positive Forman curvature $3$ everywhere.
 In this paper, we only study the class of tessellations on surfaces with positive Forman curvature, denoted by $\mathcal{FC}_{+}.$ One easily checks that the class $\mathcal{FC}_{+}$ is closed under the dual operation.  Forman proved discrete analogs of Bochner's theorem and the Bonnet-Myers theorem for ``quasi-convex'' regular CW complexes, see \cite[Theorem~2.8 and Theorem~6.3]{MR1961004}. In our setting, for a graph $G=(V,E,F)$ on a surface $S,$ they state that if $\Ric(e)>0$ for each edge $e,$ then $H^1(S,\R)=0$ and the diameter of $X$ is finite, respectively.
The above theorems yield the following result.
\begin{theorem}[\cite{MR1961004}]\label{thm:mainpfm} Let $G=(V,E,F)$ be a tessellation on a surface $S$ satisfying $\Ric(e)>0$ for all $e\in E.$ Then $G$ is finite and $S=\SP^2$ or $\R P^2.$ 
\end{theorem}

\begin{figure}[ht]\centering
  \begin{tikzpicture}[scale=2]
   \coordinate (a) at (-1,-.5);
   \coordinate (b) at (-1,.5);
   \coordinate (c) at (-1/2,0);

   \node[below] at (c) {$\ \ x_1$};
   
   \coordinate (d) at (1,-.5);
   \coordinate (e) at (1,.5);
   \coordinate (f) at (1/2,0);
   \coordinate (g) at (3/2,0);

   \node[above] at (f) {$x_2\ \ $};

   \node at (0,.4) {$f_1$};
   \node at (0,-.4) {$f_2$};
   \node[above] at (0,0) {$e$};
   
   \draw (a) -- (c) -- (b);
   \draw (d) -- (f) -- (e);
   \draw (c) -- (f) -- (g);

   \coordinate (ac) at (-3/4,-1/4);
   \coordinate (bc) at (-3/4,1/4);
   \coordinate (df) at (3/4,-1/4);
   \coordinate (fe) at (3/4,1/4);
   \coordinate (cf) at (0,0);
   \coordinate (fg) at (1,0);

  \end{tikzpicture}
 \caption{\label{fig:Forman}}
\end{figure}

 One of main difficulties for the Bonnet-Myers theorem above is that the total Forman curvature of an infinite graph on a surface is possibly infinite. So that we don't have the control for the number of edges in the case of positive Forman curvature directly. Forman circumvented the difficulty by using (combinatorial) Jacobi fields. In this paper, we give a new proof of the above result without using Jacobi fields. 
The proof strategy is as follows: for a tessellation $G$ of a surface $S,$ we obtain a formula of the Forman curvature of an edge $e,$
\begin{equation}\label{eq:rc}\Ric(e)=16-(|x_1|+|x_2|+|f_1|+|f_2|),\end{equation} where $x_1,x_2\in V$ and $f_1,f_2\in F$ such that $x_1\prec e,x_2\prec e, e\prec f_1, e\prec f_2,$ 
see Figure~\ref{fig:Forman} and Proposition~\ref{prop:FC1}.  For $G,$ we construct a \emph{medial graph} $G',$ a graph on a surface $S$, associated to $G,$ whose vertices correspond to the set of edges in $G,$ identified with the midpoints of edges, and whose faces correspond to the set of vertices and faces in $G,$ see Section~\ref{sec:pre} for details. For a graph $G$ on a surface with positive Forman curvature, by the structure of medial graphs and \eqref{eq:rc}, we will prove that the medial graph $G'$ has positive combinatorial curvature everywhere, see Proposition~\ref{prop:listpfm}. By Theorem~\ref{thm:DeVosMohar2}, $G'$ is a finite graph and the ambient space $S$ is $\SP^2$ or $\R P^2.$ This provides a new proof of the theorem. 

Via medial graphs, we classify the class $\mathcal{FC}_{+}.$ By the discrete Gauss-Bonnet theorem, we estimate that the number of vertices in a medial graph of $G$ in $\mathcal{FC}_{+}$ is at most 24, see Lemma~\ref{thm:rp2}. We denote by $\mathcal{FC}_{+}^{\SP^2}$ (resp. $\mathcal{FC}_{+}^{\R P^2}$) planar graphs (resp. graphs on $\R P^2$) in $\mathcal{FC}_{+}.$ Since the medial graph is 4-regular, we enumerate the medial graphs of planar graphs in $\mathcal{FC}_{+}^{\SP^2}$ using the algorithm of spherical quadrangulations by Brinkmann et al. \cite{MR2186681}. For the graphs on $\R P^2$ in $\mathcal{FC}_{+},$ we use the classification of  $\mathcal{FC}_{+}^{\SP^2}$ and the properties of the double covering map $\SP^2\to \R P^2.$ 
\begin{theorem}\label{thm:classPF} There are 116 graphs in $\mathcal{FC}_{+}$ up to isomorphism, which are all planar graphs, i.e. $\mathcal{FC}_{+}^{\R P^2}=\emptyset.$
\end{theorem}

In the last part of the paper, we consider the corner curvature for graphs on surfaces. 
Baues and Peyerimhoff \cite{MR1797301} introduced the so-called corner curvature for graphs on surfaces. A corner of $G=(V,E,F)$ is a pair $(x,\sigma)$ such that $x\in V,$ $\sigma\in F$ and $x$ is incident to $\sigma.$ The corner curvature of a corner $(x,\sigma)$ is defined as
$$C(x,\sigma):=\frac{1}{|x|}+\frac{1}{|\sigma|}-\frac12.$$ Baues and
Peyerimhoff \cite{MR1797301,MR2243299} proved many interesting properties, e.g. a
discrete Cartan-Hadamard Theorem, for planar graphs with non-positive
corner curvature. We denote by $$\CC:=\{G: C(x,\sigma)>0, \forall\
\mathrm{corner}\ (x,\sigma)\}$$ the class of graphs on surfaces with
positive corner curvature. It turns out the corner curvature condition
is quite strong in general, e.g. for any $G\in \CC,$ it has positive
combinatorial curvature, by
$\Phi(x)=\sum_{x\prec\sigma}C(x,\sigma)$. Since the dual operation
switches the vertices and faces of a graph on a surface, one can show
that $G\in \CC$ if and only if $G^*\in \CC;$ see \cite{MR2243299}.  By the
discrete Gauss-Bonnet theorem and some combinatorial arguments, we prove
that the number of vertices of a planar graph with positive corner
curvature, or of its dual graph, is at most 12. Then we modify
Brinkmann-McKay's program \texttt{plantri}~\cite{plantri} to classify
the set of planar graphs with positive corner curvature. It is
well-known for the community that $\CC$ is not a large class. In fact,
Keller classified the class $\CC$ using hand calculation, according to
private communication to him.

\begin{theorem}\label{thm:corner curvature} There are 22 planar graphs and 2 graphs in the projective plane with positive corner curvature up to isomorphism.
\end{theorem}

The programs are available via
\url{https://github.com/akmyh2/PFCPCnC/}.
The rest of the paper is organized as follows: In the next section, we review
embedding of graphs to the plane. In Section~\ref{sec:forman} and Section~\ref{sec:PCnC}, we explain
how to enumerate all the $\mathcal{FC}_{+}$- and $\PCnC$-graphs.

\bigskip
{\bf Acknowledgements.}
We cordially thank Gunnar Brinkmann, Beifang Chen, Nico van Cleemput, Matthias Keller,  Shiping Liu, Brendan McKay, Min Yan, Florentin M\"unch, and Norbert Peyerimhoff for many helpful discussions on curvatures on
planar graphs.  A. is supported by JSPS KAKENHI Grant Number
JP16K05247. The work was done when the first author was visiting School
of Mathematical Sciences, and Shanghai Center for Mathematical Sciences,
Fudan University. He thanks for the hospitality of these
institutes. H. is supported by NSFC, no.11831004 and no. 11926313. S. is supported by NSFC grant no. 11771083 and NSF of Fuzhou University through grant GXRC-18035.

\section{Preliminaries}\label{sec:pre}
Let $G=(V,E,F)$ be a graph on a surface $S.$ Usually, we do not distinguish $V,E,F$ with their embedding image in $S.$
 We recall basic results in graph theory.

\begin{prop}[\cite{MR2243299}]\label{prop:dualte} For a finite graph $G,$ it is a tessellation if and only if so is $G^*.$ 
\end{prop}

A graph (embedded or not) is said to be \emph{$k$-connected} if it cannot be
disconnected by removing fewer than $k$ vertices.

\begin{prop}\label{prop:tess2conn} Any finite planar tessellation $G=(V,E,F)$ is $2$-connected.
\end{prop}
\begin{proof} For any $v\in V,$ we want to show that $G$ is connected if
 we remove the vertex $v$ and its incident edges. Let
 $\{\sigma_i\}_{i=1}^N$ be the set of faces incident to $v.$ Then by the
 tessellation properties, $\cup_{i=1}^N\overline{\sigma_i}$ is
 homeomorphic to a closed disk, and its boundary $\gamma$ consists of
 edges, which is a simple closed curve. By the Jordan curve theorem,
 $\SP^2\setminus \gamma$ consists of two disjoint open disks $D_1$ and
 $D_2$. Without loss of generality, we assume that
 $\cup_{i=1}^N\overline{\sigma_i}\subset D_1.$ Then $D_2\cup \gamma$ is
 connected, which implies the connectedness of the graph $G$ by removing
 $v$ and its incident edges.
\end{proof}


For our purposes, we need the notion of the medial graph of $G,$
introduced by Steinitz \cite{Steinitz22}, see e.g. \cite[p.~104]{MR1271140}.
The \emph{medial graph} $G'=(V',E',F')$
 of $G=(V,E,F)$ is defined as follows: For each $e\in E$, choose an interior point (say midpoint) $M_e$ on the
edge $e$. Let $V'$ be the set of $M_e$ for all $e\in E$.  The vertices $M_{e_1}$ and $M_{e_2}$ are adjacent if there exist $x\in V, f\in F$ such that $x\prec e_1, x\prec e_2, e_1\prec f, e_2\prec f.$  For any $f\in F,$ draw a curve $c_{e_1,e_2}$ on $f$
between any adjacent $M_{e_1}$ and $M_{e_2}$ on the boundary of $f,$ such that $c_{e_1,e_2}$ does
not cross each other.  Let
$E'$ be the set of such curves, $F'$ be the set of connected components
of $S\setminus (V' \cup E')$.  As in Figure~\ref{fig:4reg}, each
 $f'\in F'$ corresponds to either  $f\in F$ or $x\in V$, and
$G'$ is 4-regular. 
For example, the
medial graph of tetrahedron is the graph of the octahedron.  The
medial graph of octahedron or cube is the graph of the
cuboctahedron~(Figure~\ref{fig:co}), where the cuboctahedron is an
Archimedean solid. More generally, the medial graph of the graph of a
Platonic solid $\mathbf{P}$ is the graph of a daughter
polyhedron~(Conway, Burgiel, and Goodman-Strauss~\cite[Sect.~21~``Naming
Archimedean and Catalan Polyhedra and Tilings'', p.~285]{MR2410150}) of
$\mathbf{P}$ and the dual $\mathbf{Q}$.
Medial graphs appear tacitly in \cite{Woess98}.
A tessellation $G=(V,E,F)$ satisfies a strong isoperimetric inequality, if the
average of combinatorial curvature 
	\begin{align*}
\Psi(e)=	 \sum_{\mbox{\scriptsize $\begin{array}{c} x\prec e\\ x\in V\end{array}$}}\frac{1}{|x|} + 
	\sum_{\mbox{\scriptsize $\begin{array}{c} e\prec f\\ f\in F\end{array}$}}\frac{1}{|f|} - 1 \qquad (e\in E)
	\end{align*}
of medial graph $A'$ of finite induced subgraph
$A$ of $G$ has negative supremum as $A$ grows toward $G$.

The same argument as in Proposition~2.1 in \cite{AHSW19} yields the following result.
\begin{prop}[\cite{AHSW19}]\label{prop:medtess}If $G$ is a tessellation, then so is the medial graph $G'.$ 
\end{prop}

It is obvious that a finite graph $G$ on surface and its dual graph has same medial graphs, i.e. $G'=(G^*)'.$ Given a 4-regular graph on a surface, we want to figure out whether it is a medial graph of some graph on a surface. The following theorem is very useful, and there is a canonical way to construct the ``inverse'' medial graph. 
\begin{theorem}[\protect{\cite[Theorem~2.1]{archdeacon92}}]\label{thm:arch} Any embedded $4$-regular graph whose faces can be $2$-colored is the medial graph of a unique dual pair of embedding graphs.
\end{theorem}
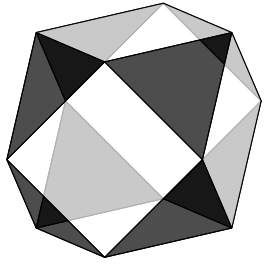
\begin{figure}
	 \centering
 \begin{tikzpicture}[line join=bevel,z=4.5,scale=2]
	     \begin{scope}[scale=1.3]
	 \coordinate (C1) at (0.15,0.15);
	 \coordinate (C2) at (0,   0.5);
	 \coordinate (C3) at (0.15,1.15);
	 \coordinate (C4) at (0.3, 0.8);

	 \coordinate (D1) at (0.8,0.3);
	 \coordinate (D2) at (1.15,0.15);
	 \coordinate (D3) at (0.5, 0);
	 
	 \coordinate (E1) at (0.8,1.3);
	 \coordinate (E2) at (1.3, 0.8);

	 \coordinate (F1) at (0.5,1);
	 \coordinate (F2) at (1, 0.5);

	  \coordinate (G1) at (1.15,1.15);

	 \draw[fill opacity=0.7, fill= white] (C1) -- (C2) -- (C3) -- (C4) -- cycle;
	 \draw[fill opacity=0.7, fill= white] (C1) -- (D1) -- (D2) -- (D3) -- cycle;
 	 \draw[fill opacity=0.7, fill= white] (C4) -- (E1) -- (E2) --
	  (D1) -- cycle;
	  \draw[fill opacity=0.7, fill=black] (C1) -- (C4) -- (D1)
	  --cycle;
	  \draw[fill opacity=0.7, fill=black] (D1) -- (E2) -- (D2)
	  --cycle;
	  \draw[fill opacity=0.7, fill=black] (E1) -- (G1) -- (E2)
	  --cycle;
	  \draw[fill opacity=0.7, fill=black] (C3) -- (C4) -- (E1) --cycle;
  	 \draw[fill opacity=0.7, fill= white] (C2) -- (F1) -- (F2) -- (D3) -- cycle;
   	 \draw[fill opacity=0.7, fill= white] (C3) -- (E1) -- (G1) -- (F1) -- cycle;
   	 \draw[fill opacity=0.7, fill= white] (D2) -- (E2) -- (G1) -- (F2) -- cycle;	 
	  \draw[fill opacity=0.7, fill=black] (C2) -- (C3) -- (F1)
	  --cycle;
	  \draw[fill opacity=0.7, fill=black] (F1) -- (F2) -- (G1) --cycle;
	  \draw[fill opacity=0.7, fill=black] (D2) -- (D3) -- (F2) --cycle;
	  \draw[fill opacity=0.7, fill=black] (C1) -- (C2) -- (D3)
	  --cycle;
	     \end{scope}
 \end{tikzpicture}
 \caption{The cuboctahedron. The graph $G'=(V',E',F')$ of cuboctahedron
 is the medial graph of the graph of cube $G=(V,E,F)$.
 $V'$ corresponds to $E$.
 The 6 white square faces of $G'$ correspond to $F$, and the other 8
 triangular faces of $G'$ to $V$. \label{fig:co}}
\end{figure}
For a 4-regular tessellation
 $H=(V',E',F')$ of $\SP^2$, the unique dual pair $(G, G^*)$ of $H$  is computed as follows:
For $f'_1,f'_2\in F'$,  we write $f'_1\sim f'_2$, if
$f'_1$ is right across from $f'_2$ with respect to some 4-valent 
  $v'\in V'$.
Then the equivalence relation $\approx\ \subseteq\ F'\times F'$
generated from $\sim$ has index two. For example, when $H$ is the graph
of cuboctahedron~(Figure~\ref{fig:co}), the equivalence classes of
$\approx$ are the set $A_1$ of eight black triangular faces and the set
$A_2$ of six white square faces. Let $\{A_1,A_2\}$ be $F'/\approx$.
We associate $G_i=(V,E,F)$ as
follows: Choose an equivalence class $A_i$, and an inner
point $P_{f'}$ for each $f'\in A_i$. Let $V$ be the set of $P_{f'}$'s. If
a face $f'_1$ of $A_i$ is right across from a face $f'_2$ of $A_i$
with respect to a 4-regular vertex $v'\in \overline{f'_1}\cap\overline{f'_2}$
of $V'$, 
i.e. $f'_1\sim f'_2$, then we draw exactly one simple curve $a_{f'_1,f'_2}$ from
$P_{f'_1}$ to $P_{f'_2}$ through $v'$. Let $E$ be the set of such
curves. Let $F$ be the set of connected components of $\SP^2\setminus
(V\cup E)$. 

If  $H$ is cuboctahedron, $G_1$~(resp. $G_2$) is a cube~(resp. regular octahedron) when $V=A_1$~(resp. $A_2$).

For a tessellation, we give a new formula for the Forman curvature.
\begin{prop}
Let $G=(V,E,F)$ be a tessellation on a surface $S.$ For any $e\in E,$ let $x_1,x_2\in V$ and $f_1,f_2\in F$ satisfy $x_1\prec e,x_2\prec e, e\prec f_1, e\prec f_2.$ Then
\begin{align}
\Ric(e)=16-(|x_1|+|x_2|+|f_1|+|f_2|).\label{prop:FC1}
\end{align}
\end{prop}
\begin{proof} {By the tessellation properties, the closed faces $\overline{f_1}$ and $\overline{f_2}$ intersect only at the closure of the edge $e.$} Let $\{e_1,e_2,\tilde{e_1},\tilde{e_2}\}\subset E\setminus\{e\}$ be the set of edges such that 
$$x_1\prec e_i\prec f_i, x_2\prec \tilde{e_i}\prec f_i,\ i=1,2.$$
Then any parallel neighbor $\tilde{e}$ of $e$ satisfies $\tilde{e}\not\in\{e,e_1,e_2,\tilde{e_1},\tilde{e_2}\}$ and $$x_1\prec \tilde{e},\ \mathrm{or}\ x_2\prec \tilde{e},\ \mathrm{or}\ \tilde{e}\prec f_1,\ \mathrm{or}\ \tilde{e}\prec f_2.$$
This yields that $$\sharp\{\mathrm{parallel\ neighbors\ of }\ e\}=|x_1|+|x_2|+|f_1|+|f_2|-12.$$
Hence by 
\eqref{eq:deffm}, we get 
$$\Ric(e)=2+2-(|x_1|+|x_2|+|f_1|+|f_2|-12).$$ This proves the proposition.
\end{proof}

For any vertex $x$ of degree $N,$ we denote by $$(|\sigma_1|,|\sigma_2|,\cdots,|\sigma_{N}|)$$ the pattern of $x,$ where $\{\sigma_i\}_{i=1}^{N}$ are the faces incident to $x$, ordered by $|\sigma_1|\leq|\sigma_2|\leq\cdots\leq|\sigma_{N}|.$ 

In the next proposition, we prove that for a graph $G$ on a surface with positive Forman curvature, the medial graph $G'$ has positive combinatorial curvature everywhere.
\begin{prop}\label{prop:listpfm} Let $G$ be a graph on surface, and $e$ be an edge with $\Ric(e)>0.$
Then the list of vertex patterns for $M_e$ in the medial graph $G'$ is given by
\begin{align}
 (3,3,3,3), (3,3,3,4), (3,3,3,5), (3,3,3,6), (3,3,4,4),  (3,3,4,5),
 (3,4,4,4).\label{vp}
\end{align} In particular, $\Phi(M_e)>0.$ Hence, for any $G\in \fmp,$ $G'$ has positive combinatorial curvature.
\end{prop}
\begin{proof} Let $x_1,x_2\in V$ and $f_1,f_2\in F$ satisfy $x_1\prec
 e,x_2\prec e, e\prec f_1, e\prec f_2.$ Then by \eqref{prop:FC1},
 $$|x_1|+|x_2|+|f_1|+|f_2|\leq 15.$$ By the structure of the medial
 graph $G',$ the facial degrees of faces incident to $M_{e}$ are given
 by $|x_1|,|x_2|,|f_1|$ and $|f_2|.$    See Figure~\ref{fig:4reg}.
Since $|x_i|\geq 3, |f_i|\geq 3,$ $i=1,2,$ we obtain the list of possible vertex patterns of $M_e$ in $G'$ as in \eqref{vp}. In particular, the case-by-case calculation yields $\Phi(M_e)>0.$ This proves the proposition.
\end{proof}
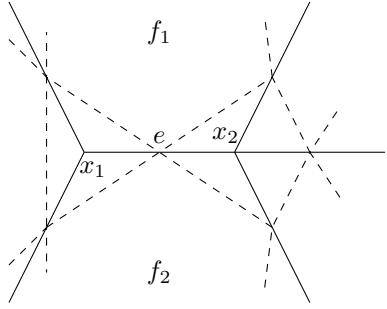
\begin{figure}[ht]\centering
  \begin{tikzpicture}[scale=2]
   \coordinate (a) at (-1,-1);
   \coordinate (b) at (-1,1);
   \coordinate (c) at (-1/2,0);

   \node[below] at (c) {$\ \ x_1$};
   
   \coordinate (d) at (1,-1);
   \coordinate (e) at (1,1);
   \coordinate (f) at (1/2,0);
   \coordinate (g) at (3/2,0);

   \node[above] at (f) {$x_2\ \ $};

   \node at (0,.8) {$f_1$};
   \node at (0,-.8) {$f_2$};
   \node[above] at (0,0) {$e$};
   
   \draw (a) -- (c) -- (b);
   \draw (d) -- (f) -- (e);
   \draw (c) -- (f) -- (g);

   \coordinate (ac) at (-3/4,-1/2);
   \coordinate (bc) at (-3/4,1/2);
   \coordinate (df) at (3/4,-1/2);
   \coordinate (fe) at (3/4,1/2);
   \coordinate (cf) at (0,0);
   \coordinate (fg) at (1,0);

   \draw[dashed] (ac)--(cf)--(bc)--cycle;
   \draw[dashed] (-1,-3/4)--(ac)--(-3/4,-.8);
   \draw[dashed] (-1,3/4)--(bc)--(-3/4,.8);
   \draw[dashed] (df)--(cf)--(fe)--(fg)--cycle;
   \draw[dashed] (.7,-.9)--(df);
   \draw[dashed] (fe)--(.7,.9);
   \draw[dashed] (1.2,-.3)--(1,0)--(1.2,.3);
  \end{tikzpicture}
 \caption{Medial graph $G'$~(dash) of a graph $G$~(solid).\label{fig:4reg}}
\end{figure}

For each $p=4,5,6$, a $p$-gonal
antiprism~(Figure~\ref{fig:ap6}) is in $\fm$.
\begin{figure}[ht]
	 \centering
\begin{tabular}[tb]{c|c|c}
  \includegraphics[scale=0.08]{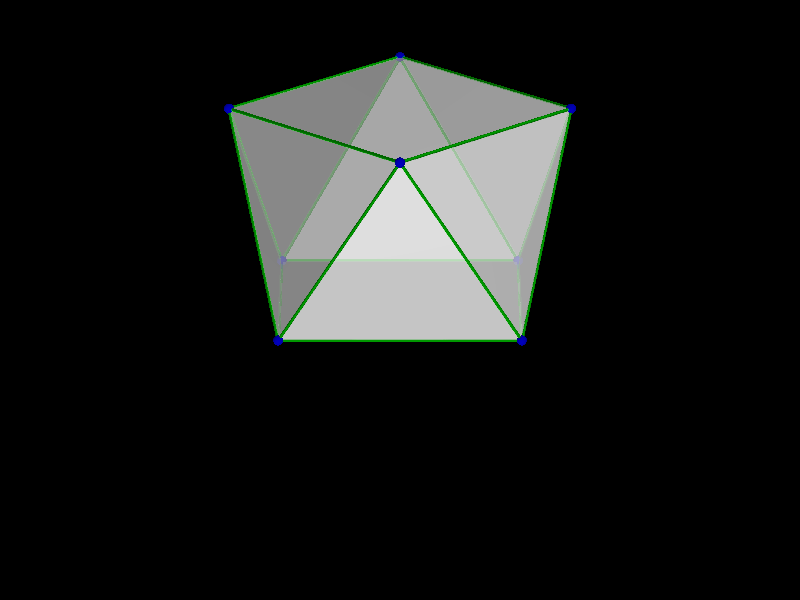}&
 \includegraphics[scale=0.08]{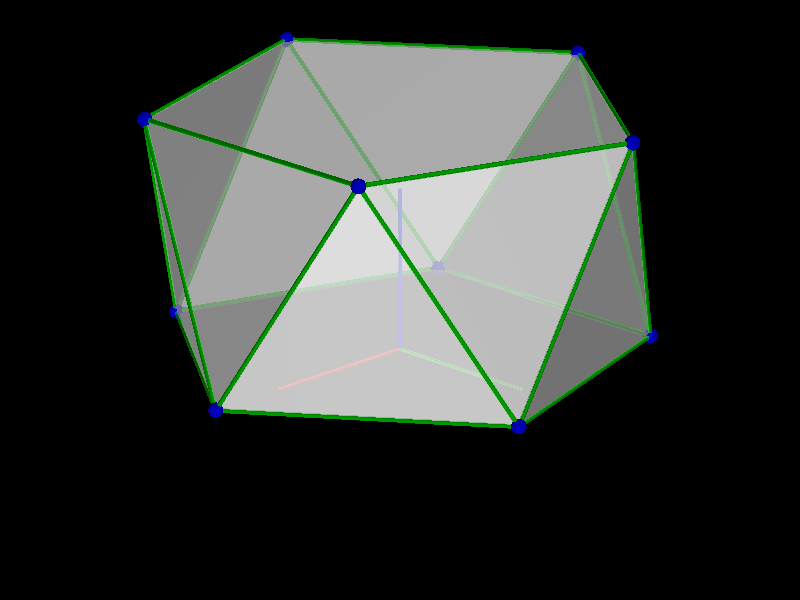} &
 \includegraphics[scale=0.08]{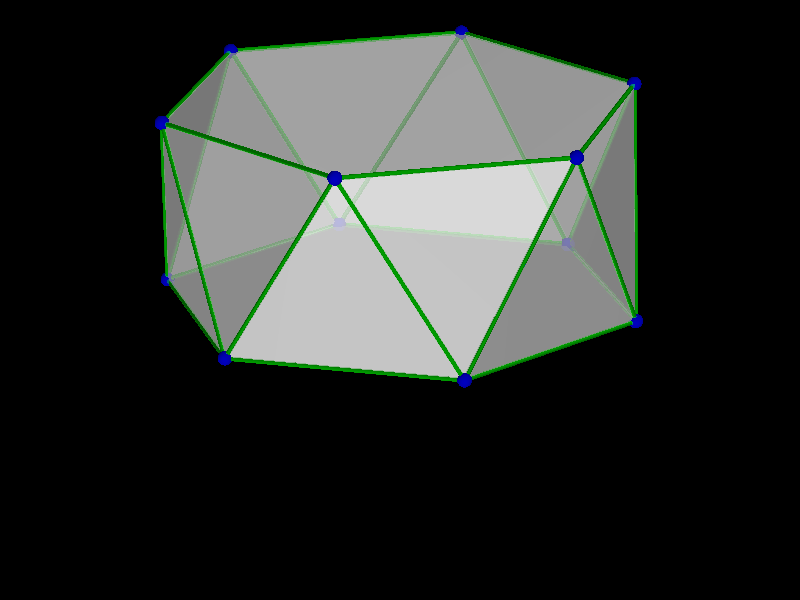} 
\end{tabular}
	 \caption{$p$-gonal antiprism~($p=4,5,6$).\label{fig:ap6}}
\end{figure}

We discuss the relation between graphs on $\SP^2$ and on $\R P^2.$
For any graph $G$ on $\R P^2,$ since $\pi: \SP^2\to \R P^2$ is a double cover, we can lift the graph structure $G$ to $\SP^2,$ denoted by $\widetilde{G}.$
\begin{theorem}\label{thm:rp2} For any tessellation $G$ on $\R P^2,$ the lifted graph $\widetilde{G}$ on $\SP^2$ is a tessellation.
\end{theorem}
\begin{proof}
We verify the tessellation properties, see the introduction, for $\widetilde{G}.$ The properties $(i),(ii),(iii)$ are trivial. It is sufficient to prove $(iv).$
That is, we need to prove the following: for any faces $\sigma_1, \sigma_2$ in $\widetilde{G},$
\begin{enumerate}[(a)]
\item if $\overline{\sigma_1}\cap\overline{\sigma_2}$ contains more than two points, then it is the closure of an edge, and 
\item if $\overline{\sigma_1}\cap\overline{\sigma_2}$ is one point, then it is a vertex.
\end{enumerate}

To prove $(a)$, suppose it is not true, then there are two faces $\sigma_1, \sigma_2$ in $\widetilde{G}$ such that 
$\overline{\sigma_1}\cap\overline{\sigma_2}$ contains at least two vertices, but it is not the closure of an edge. Let $\pi: \SP^2\to \R P^2$ be the universal cover, and $\tau_i=\pi(\sigma_i),$ $i=1,2.$ By the construction of the universal cover, we recall the following property: 
for any simply-connected open subset $U$ in $\R P^2,$ $\pi^{-1}(U)$ consists of two disjoint subsets $U_1$ and $U_2$ such that $\pi|_{U_i}: U_i\to U$ ($i=1,2$) are homeomorphisms, see e.g. the proof in \cite[pp. 64--65]{Hatcher02}. 

For any face $\tau$ in $G,$ since $\overline{\tau}$ is a closed disk, there exists an open neighborhood $A$ of $\overline{\tau}$ such that  $\pi^{-1}(A)=A_1\sqcup A_2$ and $\pi|_{A_i}$  ($i=1,2$) are homeomorphisms.

We claim that $\tau_1\neq \tau_2.$ Suppose it is not true, i.e. $\tau_1=\tau_2.$ Applying the above property for $\tau=\tau_1$, we obtain homeomorphisms $\pi|_{A_i}: A_i\to A,$ where $A$ is an open neighborhood of $\overline{\tau_1}.$ Since $\overline{\sigma_1}\cap\overline{\sigma_2}\neq \emptyset,$ 
$$\overline{\sigma_1}\cup\overline{\sigma_2}\subset A_1\ (\mathrm{or}\ A_2).$$ Without loss of generality, suppose that $\overline{\sigma_1}\cup\overline{\sigma_2}\subset A_1,$ then $(\overline{\sigma_1}\cup\overline{\sigma_2})\cap A_2=\emptyset.$ This yields a contradiction since $\pi^{-1}(\tau_1)=\sigma_1\cup\sigma_2,$ and $\pi|_{A_2}$ is a homeomorphism. This proves the claim.

For distinct points $\{p_1,p_2\}\subset \overline{\sigma_1}\cap\overline{\sigma_2},$ by the same argument as above for $\tau_1,$ we get a homeomorphism between a neighborhood of $\sigma_1$ and a neighborhood of $\tau_1.$ This implies that $\pi(p_1)\neq \pi(p_2).$ 
Hence $\overline{\tau_1}\cap\overline{\tau_2}$ consists of at least two points. By the tessellation properties of $G,$ $\overline{\tau_1}\cap\overline{\tau_2}$ is the closure of an edge in $G.$ Moreover, this yields that $\overline{\tau_1}\cup\overline{\tau_2}$ is homeomorphic to a closed disk. Then there exists an open neighborhood $W$ of $\overline{\tau_1}\cup\overline{\tau_2}$ such that  $\pi^{-1}(W)=W_1\sqcup W_2$ and $\pi|_{W_i}$  ($i=1,2$) are homeomorphisms. Since $\overline{\sigma_1}\cap\overline{\sigma_2}\neq \emptyset,$
$$\overline{\sigma_1}\cup\overline{\sigma_2}\subset W_1\ (\mathrm{or}\ W_2).$$ Since $\pi|_{W_1}$ is a homeomorphism, by the property of $\overline{\tau_1}\cap\overline{\tau_2},$ $\overline{\sigma_1}\cap\overline{\sigma_2}$ is the closure of an edge in $\widetilde{G}.$ This yields a contradiction and proves $(a).$
 
To prove $(b),$ for any faces $\sigma_1$ and $\sigma_2$ satisfying $\overline{\sigma_1}\cap\overline{\sigma_2}=\{p\},$ by the same argument as above, one can show that $\tau_1\neq\tau_2,$ where $\tau_i=\pi(\sigma_i),$ $i=1,2.$ Hence $\overline{\tau_1}\cap\overline{\tau_2}\neq \emptyset.$ We claim that $\overline{\tau_1}\cap\overline{\tau_2}$ is one point. Suppose that it contains more than two points, then it is the closure of an edge. Then the same argument as above yields that there exists a homeomorphism between an open neighborhood of $\overline{\tau_1}\cup\overline{\tau_2}$ and an open neighborhood of $\overline{\sigma_1}\cup\overline{\sigma_2}.$ This contradicts $\overline{\sigma_1}\cap\overline{\sigma_2}=\{p\}.$ This proves the claim. 

Hence by the tessellation properties of $G,$ $\overline{\tau_1}\cap\overline{\tau_2}$ is a vertex. Then $\overline{\tau_1}\cup\overline{\tau_2}$ is simply-connected and there exists a simply-connected, open neighborhood $W$ of $\overline{\tau_1}\cup\overline{\tau_2}.$ By the same argument as above, $W$ is homeomorphic to an open neighborhood of $\overline{\sigma_1}\cup\overline{\sigma_2}.$ This implies $p$ is a vertex, and yields the result $(b).$
This proves the theorem.
\end{proof}

By the above theorem, in order to classify the class $\fmp^{\R P^2}$, it is sufficient to classify the class $\fm$ and to figure out the candidates whose projections into $\R P^2$ are tessellations.

\subsection{Embedding}

To fully represent the embedded graph, we need
both the abstract graph $(V,E)$
  and the cyclic edge orders.  In the case of a graph with no multiedges, it is conventional
  to give both at once by listing neighbors in anti-clockwise order.

An \emph{adjacency list} of a graph $G$ is, by definition, a list of
pairs of a vertex $x$ and the counter-clockwise cyclic list $N_x$ of
vertices adjacent to $x$. \emph{A mirror image} of a graph $G$ is, by
definition, a graph such that an adjacency list is a list of pairs of
 a vertex $x$ and \emph{reversed} $N_x$.

 \begin{figure}[ht]\centering
  \begin{tabular}{c|c}
\begin{tikzpicture}
	       \coordinate (a102) at (1,1);
	       \coordinate (b102) at (0,1);
	       \coordinate (c102) at (.4,.6);
	       \coordinate (d102) at (.7,.3);
	       \coordinate (e102) at (1,0);
	       \coordinate (f102) at (0,0);

 \node at (a102) {$a$};
 \node at (b102) {$b$};
 \node at (c102) {$c$};
 \node at (d102) {$d$};
 \node at (e102) {$e$};
  \node at (f102) {$f$};
\draw (a102) -- (b102);
\draw (a102) -- (c102);
\draw (a102) -- (d102);
\draw (a102) -- (e102);

\draw (b102) -- (c102);
\draw (b102) -- (f102);

\draw (c102) -- (d102);

\draw (d102) -- (e102);
\draw (d102) -- (f102);

\draw (e102) -- (f102);
\end{tikzpicture} & 
\begin{tikzpicture}[xscale=-1]
	       \coordinate (a102) at (1,1);
	       \coordinate (b102) at (0,1);
	       \coordinate (c102) at (.4,.6);
	       \coordinate (d102) at (.7,.3);
	       \coordinate (e102) at (1,0);
	       \coordinate (f102) at (0,0);

 \node at (a102) {$a$};
 \node at (b102) {$b$};
 \node at (c102) {$c$};
 \node at (d102) {$d$};
 \node at (e102) {$e$};
  \node at (f102) {$f$};
\draw (a102) -- (b102);
\draw (a102) -- (c102);
\draw (a102) -- (d102);
\draw (a102) -- (e102);

\draw (b102) -- (c102);
\draw (b102) -- (f102);

\draw (c102) -- (d102);

\draw (d102) -- (e102);
\draw (d102) -- (f102);

\draw (e102) -- (f102);
   \end{tikzpicture}\\

\scriptsize  a:bcde,\ b:caf,\ c:dab,\ d:eacf,\ e:dfa,\ f:bed
   &
\scriptsize a:edcb,\ b:fac,\ c:bad,\ d:fcae,\ e:afd,\ f:deb
\end{tabular}
\caption{A graph and an adjacency list. The left part and the right part
  are  ``mirror image'' to each other.\label{fig:eg}}
\end{figure}
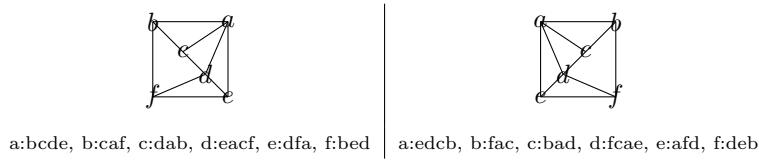
 Let $\G$ be the set of connected, simple, planar
graphs $G=(V,E,F)$ such that  both of face degrees and vertex degrees are  at least
3 and finite.

Given $G,G'\in\G$.  $G$ is said to be \emph{order-preserving
isomorphic}~(OP-isomorphic, for short) to $G'$, if there is a bijection
$\varphi$ from $G$ to $G'$ that preserves incidence relation for
vertices, edges, and faces, the cyclic edge-orderings around the
vertices. $G$ is said to be \emph{order-reversing
isomorphic}~(OR-isomorphic, for short) to $G'$, if $G$ is OP-isomorphic
to a mirror image of $G'$. $G$ is said to be \emph{isomorphic} to $G'$, 
if they are OP-isomorphic or OR-isomorphic~\cite{plantri}.

\begin{quote}
  ``plantri is a program that generates certain types of graphs that are
  imbedded on the sphere.
 
  Exactly one member of each isomorphism class is output, using an amount
  of memory almost independent of the number of graphs produced.  This,
  together with the exceptionally fast operation and careful validation,
  makes the program suitable for processing very large numbers of
  graphs.

    Exactly one member of each isomorphism class is output, using an amount
  of memory almost independent of the number of graphs produced.  This,
  together with the exceptionally fast operation and careful validation,
  makes the program suitable for processing very large numbers of graphs.

  Isomorphisms are defined with respect to the imbeddings, so in some
  cases outputs may be isomorphic as abstract graphs.''~(plantri-guide.txt~\cite{plantri})
\end{quote}

Below, the figures of graphs $G=(V,E,F)$ are based on some embeddings $\varphi$'s 
to the plane such that
\begin{itemize}
 \item the edges of $G$ become mutually noncrossing line
      segments of the figures; and
      
 \item  the cyclic ordering of edges of $G$ incident to each vertex
$v\in V$ becomes the counter-clockwise ordering of the line segments around the
point $\varphi(v)$.
\end{itemize}

For example, in Figure~\ref{fig:Schlegel}, the left is 3d-representation
of a graph as a convex polytope, and the right is an embedding of the
same graph such that the edges become mutually noncrossing line segments
and the cyclic edge-ordering around each vertex is preserved.
 \begin{figure}[ht]\centering
  \begin{tabular}{cc}
   \includegraphics[scale=0.3]{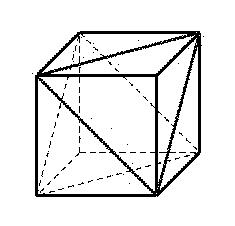} & 
			             \begin{tikzpicture}[yshift=1cm,scale=0.5]
\foreach \x in {1,...,3}
{
\draw[rotate=120*\x]
  (90:2) -- (210:2)--(150:0.5)--cycle;
\draw[rotate=120*\x]
  (0:0) --  (150:0.5) -- (270:0.5);
					 }
					 \draw (210:2) -- (90:2.5) -- (330:2);
					 \draw (90:2) -- (90:2.5);
					\end{tikzpicture}
\end{tabular}
   \caption{\label{fig:Schlegel}}
 \end{figure}
$J_n$ stands for the $n$-th Johnson solid~\cite{MR0185507,MR0240719}.

\section{Graphs on surfaces with positive Forman curvature}\label{sec:forman}
In this section, we give a new proof of Theorem~\ref{thm:mainpfm} and classify the class $\fmp.$

\begin{proof}[Proof of Theorem~\ref{thm:mainpfm}] Let $G'$ be the medial graph of $G.$ Then $G'$ is a tessellation of $S,$ see Proposition~\ref{prop:medtess}. Moreover, $G'$ has positive combinatorial curvature by Proposition~\ref{prop:listpfm}. By DeVos-Mohar's solution of Higuchi's conjecture, Theorem~\ref{thm:DeVosMohar2}, we prove the theorem.
\end{proof}

In the following, we prove Theorem~\ref{thm:classPF} by classifying the
class $\fmp.$ We need the following lemmas.
\begin{lemma} \label{lem:v1}
   For any medial graph $G'=(V',E',F')$ of $G\in \fmp$, the vertex
   pattern of any $v$ is given in \eqref{vp}, and $\#V'\le24$.
\end{lemma}

\begin{proof}The list of vertex patterns follows from Proposition~\ref{prop:listpfm}. Moreover, for any vertex $v,$ whose vertex pattern is given in the list, $\Phi(v)\geq 1/12.$  By the discrete Gauss-Bonnet theorem \eqref{eq:GBthm}, 
$$\frac{1}{12}\sharp V'\leq \sum_{v\in V'} \Phi(v)\leq 2.$$ This yields the result.
\end{proof}

A graph $G=(V,E,F)$ on $\SP^2$ is called a \emph{spherical quadrangulation} if $|f|=4$ for any $f\in F.$ For any $f\in F,$ the \emph{face pattern} of $f$ is given by $(|v_1|,|v_2|,|v_3|,|v_4|),$ where $\{v_i\}_{i=1}^4$ are vertices incident to $f$ and $|v_1|\leq |v_2|\leq |v_3|\leq |v_4|.$ Let $\mathcal{Q}$ be the set of $2$-connected, simple, spherical quadrangulations $G=(V,E,F)$ with $\sharp F\leq 24,$ whose face patterns are in the list \eqref{vp}.
\begin{lemma} \label{lem:v2} For any $G\in \fm$, $(G')^*\in \mathcal{Q}.$ \end{lemma}
\begin{proof} Since $G'$ is $4$-regular, the lemma follows from Proposition~\ref{prop:dualte}, Proposition~\ref{prop:tess2conn}, Proposition~\ref{prop:medtess}, and Lemma~\ref{lem:v1}.
\end{proof}

\begin{proof}[Proof of Theorem~\ref{thm:classPF}]
We first classify the finite set $\fm.$  
By Lemma~\ref{lem:v2}, for any $G\in \fm,$ $(G')^*\in \mathcal{Q}.$ We enumerate the class $\mathcal{Q},$ which serves the set of candidates of $(G')^*$ for some $G\in \fm,$  by a computer program. Then using the duality, we obtain $G'$ from the output, and by Theorem~\ref{thm:arch}, we construct the dual pair $(G,G^*),$ whose medial graph is $G',$ in the canonical way (described after Theorem~\ref{thm:arch}). This gives the classification of $\fm.$

The program is
a modification of Brinkmann-McKay's \texttt{plantri}.  \texttt{Plantri}
is a program that generates certain types of graphs that are embedded on
the sphere.  Exactly one member of each isomorphism class is output,
using an amount of memory almost independent of the number of graphs
produced.  Here isomorphisms are graph isomorphisms which also take
embedding to the plane~(sphere) into account.  In particular, \texttt{plantri}
enumerates quickly all 4-regular $H$, such as medial graphs, by an algorithm developed
in~\cite{MR2186681}.  Every
spherical quadrangulation of vertices more than 4 is obtained from a pseudo-double wheel through finite applications of two local expansions~(see
Figure~\ref{p_expansion}).
\begin{figure}[ht]\centering
\begin{tabular}{c}
\includegraphics[scale=.2]{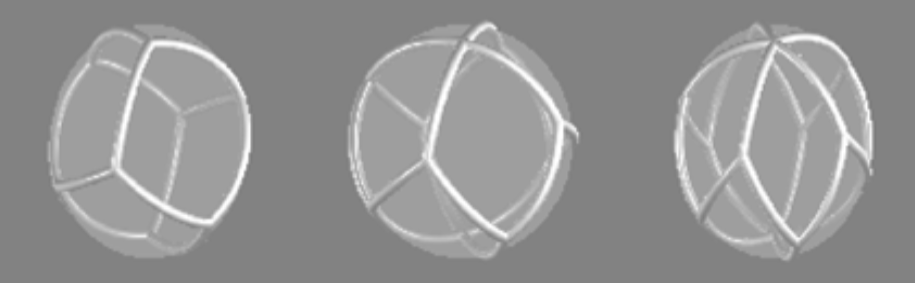}\\
\includegraphics[scale=.1]{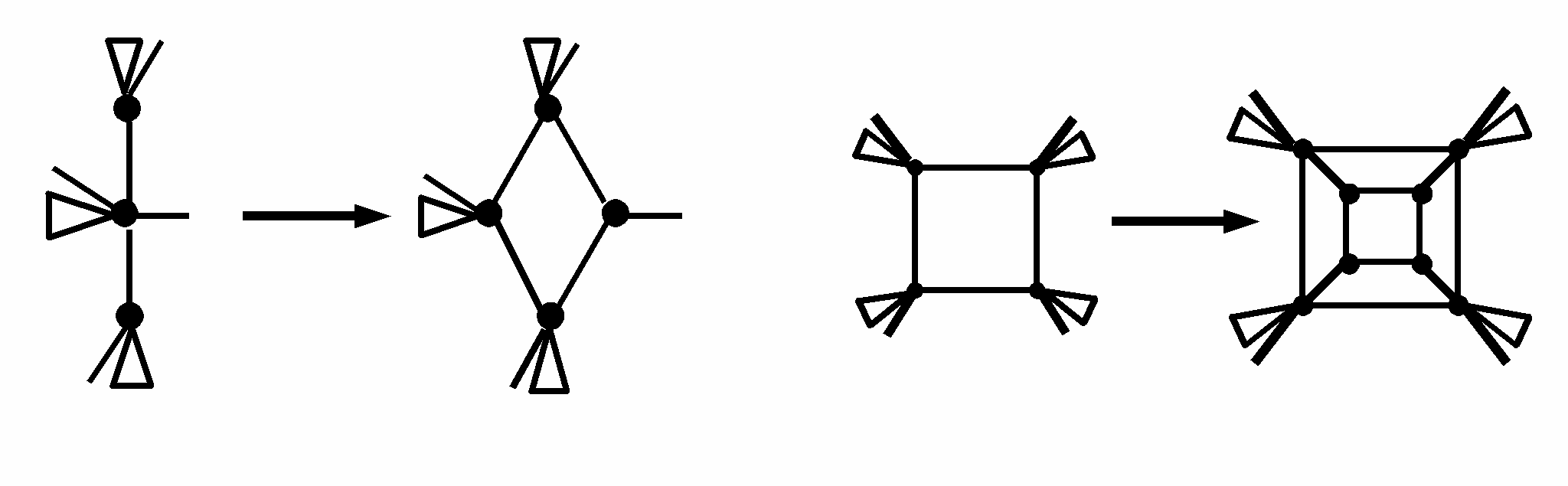}\end{tabular}

\caption{$p$-gonal pseudo-double wheels~($p=3,4,5$), and two expansions of spherical quadrangulations to increase the number
 of faces (Brinkmann et al.~\cite{MR2186681}). 
 \label{p_expansion}}
\end{figure}

We added to \texttt{plantri}, the following:
 \begin{enumerate}

  \item the restriction of vertex pattern of $H$ to be one of \eqref{vp};
	
  \item the computation of the dual pair $(G,G^*)$ such that $H$ is the
	medial graph of $(G,G^*)$.
 \end{enumerate}

As a result, we found that the set of medial graphs $G'$ of
$\fm$-graphs consists of 73 simple, 2-connected, 4-regular, planar graphs.
For each graph $G'$, we generate the dual pair $(G,G^*)$ 
such that the $G'$ is the medial graph of $(G,G^*)$.  We check that all obtained pairs $(G,G^*)$ are tessellations on $\SP^2.$ There are exactly 30 self-dual pairs. We say the pair
$(G,G^*)$ is \emph{self-dual} if $G=G^*$. Hence $\fm$ consists of
$30+(73-30)\times2=116$ graphs.


Among the 116 $\fm$-graphs, only five graphs $\fm$-E$10$-1V$6$,
$\fm$-E$12$-8V$7$, $\fm$-E12-8V7Dual, $\fm$-E$13$-1V$7$, and $\fm$-E$13$-1V$8$Dual are
not 3-connected. Hence, the five graphs are not skeletons of convex
polyhedra, by Steinitz's theorem~\cite{MR1311028}.

 \begin{table}[ht]\centering
     \begin{tabular}{|c|c|c|c || c|c|}
      $\#E$ & \# & $\#V$ & $G=(V,E,F)$ & $\#F$  & $G^*$  \\
\hline
 6 & 1 & 4 & Regular tetrahedron &&\\
 8 & 1 & 5 & Square pyramid  ($J_1$) &&\\
 9 & 1 & 5 & $3$-gonal bipyramid ($J_{12}$)  &
		      6 & $3$-gonal prism \\
&&&&      &\\
10$^\dag$ & 1 & 6 &
\begin{tabular}{c}
 	      \begin{tikzpicture}[scale=0.5]

\coordinate (a103) at (1,-1);
\coordinate (b103) at (1,1);
\coordinate (c103) at (-1,1);
\coordinate (d103) at (.3,.3);
\coordinate (e103) at (-1,-1);
\coordinate (f103) at (-.3,-.3);

\draw (a103) -- (b103);
\draw (a103) -- (d103);
\draw (a103) -- (e103);
\draw (a103) -- (f103);

\draw (b103) -- (c103);
\draw (b103) -- (d103);

\draw (c103) -- (d103);
\draw (c103) -- (e103);
\draw (c103) -- (f103);

\draw (e103) -- (f103);

	      \end{tikzpicture}
		  \end{tabular}
%
		  &&\\
      10 & 2 & 6 &  $5$-gonal pyramid ($J_2$)&&\\
&&&   &  &\\
10 & 3 & 6 & \begin{tabular}{c}
					  \begin{tikzpicture}
	       \coordinate (a102) at (1,1);
	       \coordinate (b102) at (0,1);
	       \coordinate (c102) at (.4,.6);
	       \coordinate (d102) at (.7,.3);
	       \coordinate (e102) at (1,0);
	       \coordinate (f102) at (0,0);

\draw (a102) -- (b102);
\draw (a102) -- (c102);
\draw (a102) -- (d102);
\draw (a102) -- (e102);

\draw (b102) -- (c102);
\draw (b102) -- (f102);

\draw (c102) -- (d102);

\draw (d102) -- (e102);
\draw (d102) -- (f102);

\draw (e102) -- (f102);
\end{tikzpicture}
					 \end{tabular}
		  &&\\
11 & 1 & 6 & \begin{tabular}{c}
	       \begin{tikzpicture}[scale=0.8]
	       \coordinate (a) at (90:1);
	       \coordinate (b) at (330:1);
	       \coordinate (c) at (330:.5);
	       \coordinate (d) at (90:.5);
	       \coordinate (e) at (210:.5);
	       \coordinate (f) at (210:1);
	       \draw (a)--(b)--(f)--cycle;
	       \draw (c)--(d)--(e)--cycle;
	       \draw (f)--(e)--(a)--(c)--(b);
	       \draw (a)--(d);
	       \end{tikzpicture}
	     \end{tabular} &
		      7 & \begin{tabular}{c}
	       \begin{tikzpicture}[scale=0.5]
		\coordinate (a) at (162:1);
	       \coordinate (b) at (234:1);
	       \coordinate (c) at (270:.5);
	       \coordinate (d) at (0:0);
	       \coordinate (e) at (90:1);
	       \coordinate (f) at (18:1);
	       \coordinate (g) at (306:1);
	       \draw (a)--(b)--(c)--(d)--(e)--cycle;
	       \draw (c)--(g)--(f)--(d)--cycle;
	       \draw (g)--(b);
	       \draw (a)--(d);
	       \draw (e)--(f);
	       \end{tikzpicture}
			  \end{tabular} 
          \\
11 & 2  &
         6 & \begin{tabular}{c}
	       \begin{tikzpicture}[scale=0.5]
	       \coordinate (a) at (-.5,-.5);
	       \coordinate (b) at (1,-1);
	       \coordinate (c) at (-1,-1);
	       \coordinate (d) at (-1,1);
	       \coordinate (e) at (.5,.5);
	       \coordinate (f) at (1,1);

	       \draw (d)--(a)--(e)--(d)--(c)--(b)--(f)--(d);
	       \draw (c)--(a)--(b)--(e);
	       \draw (f)--(e);
	       \end{tikzpicture}
	     \end{tabular}
		  & 7 & \begin{tabular}{c}
	       \begin{tikzpicture}[scale=0.5]
	       \coordinate (a) at (.5,-.5);
	       \coordinate (b) at (1,-1);
	       \coordinate (c) at (1,1);
	       \coordinate (d) at (-1,1);
	       \coordinate (e) at (.5,.5);
	       \coordinate (f) at (-1,-1);
	       \coordinate (g) at (-.5,-.5);

1	       \draw (a)--(e)--(d)--(g)--(a)--(b)--(c)--(d)--(f)--(b);
	       \draw (c)--(e);
	       \draw (f)--(g);
	       \end{tikzpicture}
	     \end{tabular} \\
     \end{tabular}
  \caption{ 
11  $\fm$-graphs such that $6\le\#E\le11$.\label{tbl:6891011}}
\end{table}
\begin{table}[ht]\centering
     \begin{tabular}{|c|c|c|c || c|c|}
      $\#E$ & \# & $\#V$ & $G=(V,E,F)$ & $\#F$ & $G^*$\\
      \hline
      12 & 1 & 6 & Regular octahedron & 8 & Cube  \\
12 & 2 & 7 & Hexagonal pyramid \\
12 & 3 & 7 & \begin{tabular}{c}
		\begin{tikzpicture}[scale=.7]



\coordinate (a122Dual) at (0:0);
\coordinate (b122Dual) at (18:1);
\coordinate (c122Dual) at (306:1);
\coordinate (d122Dual) at (234:1);
\coordinate (e122Dual) at (162:1);
\coordinate (f122Dual) at (90:1);
\coordinate (g122Dual) at (210:.5);

\draw (a122Dual) -- (b122Dual);
\draw (a122Dual) -- (c122Dual);
\draw (a122Dual) -- (e122Dual);
\draw (a122Dual) -- (f122Dual);
\draw (a122Dual) -- (g122Dual);

\draw (b122Dual) -- (c122Dual);
\draw (b122Dual) -- (f122Dual);

\draw (c122Dual) -- (d122Dual);

\draw (d122Dual) -- (e122Dual);
\draw (d122Dual) -- (g122Dual);

\draw (e122Dual) -- (f122Dual);
\draw (e122Dual) -- (g122Dual);



		 %
		\end{tikzpicture}
	     \end{tabular} \\
	      12 & 4 & 7 & \begin{tabular}{c}
			     \begin{tikzpicture}[scale=.5]


\coordinate (a123) at (0,-.5);
\coordinate (b123) at (-1,-1);
\coordinate (c123) at (-.5,0);
\coordinate (d123) at (-1,1);
\coordinate (e123) at (1,1);
\coordinate (f123) at (1,-1);
\coordinate (g123) at (.6,.6);

\draw (a123) -- (b123);
\draw (a123) -- (f123);
\draw (a123) -- (g123);

\draw (b123) -- (c123);
\draw (b123) -- (d123);
\draw (b123) -- (f123);

\draw (c123) -- (d123);
\draw (c123) -- (g123);

\draw (d123) -- (e123);

\draw (e123) -- (f123);
\draw (e123) -- (g123);

\draw (f123) -- (g123);
			     \end{tikzpicture}
			   \end{tabular}  \\
12 & 5 & 7 & \begin{tabular}{c}
	      \begin{tikzpicture}[scale=0.6]
	       \elongatedpyramid{3}
	      \end{tikzpicture}\\
	      (Elongated triangular\\
	      pyramid ($J_7$))
	     \end{tabular}  \\
12 & 6 & 7 & \begin{tabular}{c}
	       \begin{tikzpicture}[scale=0.5]


\coordinate (a125) at (1,-1);
\coordinate (b125) at (.5,-.5);
\coordinate (c125) at (1,1);
\coordinate (d125) at (0,0);
\coordinate (e125) at (-.5,.5);
\coordinate (f125) at (-1,1);
\coordinate (g125) at (-1,-1);

\draw (a125) -- (b125);
\draw (a125) -- (c125);
\draw (a125) -- (g125);

\draw (b125) -- (c125);
\draw (b125) -- (d125);

\draw (c125) -- (d125);
\draw (c125) -- (f125);

\draw (d125) -- (e125);
\draw (d125) -- (g125);

\draw (e125) -- (f125);
\draw (e125) -- (g125);

\draw (f125) -- (g125);

	       \end{tikzpicture}
	     	     \end{tabular} \\
	      12 & 7 & 7 & \begin{tabular}{c}
			    \begin{tikzpicture}[scale=0.6]
			     \trapezocupola{3}{.3}
			    \end{tikzpicture}
			   \end{tabular} \\
      12$^\dag$ & 8 & 7 &\begin{tabular}{c}
	     \begin{tikzpicture}[scale=.7]



\coordinate (a129) at (210:1);
\coordinate (b129) at (180:.3);
\coordinate (c129) at (90:.3);
\coordinate (d129) at (0:.3);
\coordinate (e129) at (270:.3);
\coordinate (f129) at (330:1);
\coordinate (g129) at (90:1);

\draw (a129) -- (b129);
\draw (a129) -- (f129);
\draw (a129) -- (g129);

\draw (b129) -- (c129);
\draw (b129) -- (e129);
\draw (b129) -- (g129);

\draw (c129) -- (d129);
\draw (c129) -- (e129);

\draw (d129) -- (e129);
\draw (d129) -- (f129);
\draw (d129) -- (g129);

\draw (f129) -- (g129);
	     \end{tikzpicture}\\
	    \end{tabular}   &
	  7 &
\begin{tabular}{c}
		 \begin{tikzpicture}[scale=.5]



\coordinate (a129Dual) at (270:.3);
\coordinate (b129Dual) at (150:.3);
\coordinate (c129Dual) at (30:.3);
\coordinate (d129Dual) at (90:1);
\coordinate (e129Dual) at (180:1);
\coordinate (f129Dual) at (270:1);
\coordinate (g129Dual) at (0:1);

\draw (a129Dual) -- (b129Dual);
\draw (a129Dual) -- (c129Dual);
\draw (a129Dual) -- (e129Dual);
\draw (a129Dual) -- (f129Dual);
\draw (a129Dual) -- (g129Dual);

\draw (b129Dual) -- (c129Dual);
\draw (b129Dual) -- (d129Dual);

\draw (c129Dual) -- (d129Dual);

\draw (d129Dual) -- (e129Dual);
\draw (d129Dual) -- (g129Dual);

\draw (e129Dual) -- (f129Dual);

\draw (f129Dual) -- (g129Dual);

\end{tikzpicture}
\end{tabular}
	      \\
	      12 & 9 & 7 & \begin{tabular}{c}
			     \begin{tikzpicture}[scale=0.5]
			      \begin{scope}\draw (90:1) --  (162:1) -- (234:1) --
			       (-.5,0) -- cycle; \draw (-.5,0) -- (162:1);
			      \end{scope}
		  \begin{scope}[xscale=-1]\draw (90:1) --  (162:1) -- (234:1) --
		   (-.5,0) -- cycle; \draw (-.5,0) -- (162:1);
		  \end{scope}
			     \draw (-.5,0) -- (.5,0);
			     \draw (234:1) -- (306:1);
			     \end{tikzpicture}
			   \end{tabular}   &
         7 & \begin{tabular}{c}
			    \begin{tikzpicture}[scale=0.5]
		 \draw (-1,1) 
		 -- (1,1)     
		 -- (1,-1)    
		 -- (.5,.5)   
		 -- (-.5,.5)  
		 --cycle;     
			     \draw (-.5,-.5) 
			     -- (1,-1)       
			     -- (-.5,.5)     
			     -- (-.5,-.5)    
			     -- (-1,-1)   
			     -- (1,-1);    
			     \draw (-1,-1) 
			     -- (-1,1);    
			     \draw (1,1)   
			     -- (.5,.5);   
			    \end{tikzpicture}
			   \end{tabular}  \\
     \end{tabular}
 \caption{12 12-edge $\fm$-graphs.\label{tbl:12}}
\end{table}

\def\thirteenedge{
\begin{tabular}{|c|c|c|c || c|c|}
 $\#E$ & \# & $\#V$ & $G=(V,E,F)$ & $\#F$  & $G^*$ \\
 \hline
&&&&& \\
 13$^\dag$ & 1 & 7 &\begin{tabular}{c}
 \begin{tikzpicture}[scale=.5]
\coordinate (a135) at (-1,-1);
\coordinate (b135) at (-1,1);
\coordinate (c135) at (1,1);
\coordinate (d135) at (.5,.5);
\coordinate (e135) at (1,-1);
\coordinate (f135) at (0,0);
\coordinate (g135) at (-.5,-.5);

\draw (a135) -- (b135);
\draw (a135) -- (e135);
\draw (a135) -- (g135);

\draw (b135) -- (c135);
\draw (b135) -- (d135);
\draw (b135) -- (f135);
\draw (b135) -- (g135);

\draw (c135) -- (d135);
\draw (c135) -- (e135);

\draw (d135) -- (e135);

\draw (e135) -- (f135);
\draw (e135) -- (g135);

\draw (f135) -- (g135);
\end{tikzpicture}				   
				  \end{tabular}  
&		      8 & \begin{tabular}{c}
			   \begin{tikzpicture}[scale=.5]


\coordinate (a135Dual) at (-1,1);
\coordinate (b135Dual) at (1,1);
\coordinate (c135Dual) at (1,-1);
\coordinate (d135Dual) at (-1,-1);
\coordinate (e135Dual) at (0,-.5);
\coordinate (f135Dual) at (.5,0);
\coordinate (g135Dual) at (0,.5);
\coordinate (h135Dual) at (-.5,0);

\draw (a135Dual) -- (b135Dual);
\draw (a135Dual) -- (d135Dual);
\draw (a135Dual) -- (g135Dual);

\draw (b135Dual) -- (c135Dual);
\draw (b135Dual) -- (g135Dual);

\draw (c135Dual) -- (d135Dual);
\draw (c135Dual) -- (e135Dual);

\draw (d135Dual) -- (e135Dual);

\draw (e135Dual) -- (f135Dual);
\draw (e135Dual) -- (h135Dual);

\draw (f135Dual) -- (g135Dual);
\draw (f135Dual) -- (h135Dual);

\draw (g135Dual) -- (h135Dual);
			   \end{tikzpicture}
			  \end{tabular}
		     \\
13 & 2 & 7 & \begin{tabular}{c}\\
	\begin{tikzpicture}[scale=0.6]
	 \coordinate (a) at (234:1);
	 \coordinate (b) at (162:1);
	 \coordinate (c) at (90:1);
	 \coordinate (d) at (18:1);
	 \coordinate (e) at (-54:1);
	 \coordinate (f) at (-90:.3);
	 \coordinate (g) at (90:.3);
	 \draw (a)--(b)--(c)--(d)--(e)--cycle;
	 \draw (b)--(f)--(d)--(g)--cycle;
	 \draw (a)--(f)--(e);
	 \draw (c)--(g)--(f);
		\end{tikzpicture}
	     \end{tabular}  &
         8 & \begin{tabular}{c}\\
	\begin{tikzpicture}[scale=0.6]
	 \coordinate (a) at (90:.3);
	 \coordinate (b) at (90:1);
	 \coordinate (c) at (18:1);
	 \coordinate (d) at (-54:1);
	 \coordinate (e) at (234:1);
	 \coordinate (f) at (224:.3);
	 \coordinate (g) at (-44:.3);
	 \coordinate (h) at (162:1);
	 \draw (a)--(b)--(c)--(d)--(e)--(h)--cycle;
	 \draw (a)--(c);
	 \draw (b)--(h);
	 \draw (a)--(f)--(g)--cycle;
	 \draw (g)--(d);
	 \draw (f)--(e);
		\end{tikzpicture}
	     \end{tabular} \\
13 & 3 & 7 & \begin{tabular}{c}
	      \begin{tikzpicture}[scale=0.7]
	       \prism{3}{.2}
	       \draw (-90:.3)--(210:.2);
	       \draw (-90:.3)--(210:1);
	       \draw (-90:.3)--(330:.2);
	       \draw (-90:.3)--(330:1);
	      \end{tikzpicture}
	     \end{tabular}   &
         8 & \begin{tabular}{c}
	       \begin{tikzpicture}[scale=0.7]
		\begin{scope}[rotate=45]
	       \prism{4}{.5}
		\end{scope}
	       \draw (-45:.5)--(135:.5);
	       \end{tikzpicture}
	     \end{tabular}  \\
13 & 4 & 7 & \begin{tabular}{c}
	       \begin{tikzpicture}[scale=0.5]
		\draw (1,1)      
		 --(1,-1)    
		 --(-1,-1)   
		 --(-1,1)    
		 --cycle;
		\draw (.6,-.6) 
         	 -- (1,1) 
		 -- (0,.5)   
		 -- (-1,1)   
		 -- (-.5,0)   
		 -- (.6,-.6) 
		-- (1,-1);  
		\draw (-1,-1) 
		-- (.6,-.6) 
		-- (0,.5);    
		\draw (-1,-1) 
		-- (-.5,0);   
	       \end{tikzpicture}
	     \end{tabular}  &
         8 & \begin{tabular}{c}
		\begin{tikzpicture}[scale=0.5]
		 \foreach \x in {1,...,5}{
		 \draw[rotate=72*\x] (90:1)
		 --(162:1);};
		 \coordinate (b) at (130:.5);
		 \coordinate (g) at (50:.5);
		 \draw (b) -- (162:1);
		 \draw (g) -- (18:1);
		\draw (90:1) -- (b)--(0,-.2)--(g)--cycle;
 		\draw (306:1)-- (0,-.2) -- (234:1);
		\end{tikzpicture}
	     \end{tabular} \\
13 & 5 &7 & \begin{tabular}{c}
		 \begin{tikzpicture}[scale=0.7]

\coordinate (a134) at (90:.3);
\coordinate (b134) at (210:.3);
\coordinate (c134) at (270:.3);
\coordinate (d134) at (330:.3);
\coordinate (e134) at (330:1);
\coordinate (f134) at (210:1);
\coordinate (g134) at (90:1);

\draw (a134) -- (b134);
\draw (a134) -- (c134);
\draw (a134) -- (d134);
\draw (a134) -- (g134);

\draw (b134) -- (c134);
\draw (b134) -- (f134);

\draw (c134) -- (d134);
\draw (c134) -- (e134);
\draw (c134) -- (f134);

\draw (d134) -- (e134);

\draw (e134) -- (f134);
\draw (e134) -- (g134);

\draw (f134) -- (g134);

%
%
		 \end{tikzpicture}
	    \end{tabular}
	     & 8 & \begin{tabular}{c}
	       \begin{tikzpicture}[scale=0.6]
	       \coordinate (a134Dual) at (232:1);
	       \coordinate (b134Dual) at (210:.3);
	       \coordinate (c134Dual) at (330:.3);
	       \coordinate (d134Dual) at (90:.3);
	       \coordinate (e134Dual) at (90:1);
	       \coordinate (f134Dual) at (162:1);
	       \coordinate (g134Dual) at (18:1);
	       \coordinate (h134Dual) at (306:1);
	       \draw (a134Dual)--(h134Dual)--(g134Dual)--(e134Dual)--(f134Dual)--cycle;
	       \draw (c134Dual)--(b134Dual)--(d134Dual)--cycle;
	       \draw (g134Dual)--(c134Dual)--(h134Dual);
	       \draw (f134Dual)--(b134Dual)--(a134Dual);
	       \draw (d134Dual)--(e134Dual);
	       \end{tikzpicture}
	     \end{tabular}  
         \\
13 & 6 & 7 &\begin{tabular}{c}
	       \begin{tikzpicture}[scale=0.7]
	       \coordinate (a) at (-.3,.3);
	       \coordinate (b) at (-1,0);
	       \coordinate (c) at (0,1);
	       \coordinate (d) at (1,0);
	       \coordinate (e) at (.3,.3);
	       \coordinate (f) at (0,-1);
	       \coordinate (g) at (-.3,-.3);
	       \draw (b)--(c)--(d)--(f)--cycle;
	       \draw (a)--(b)--(g)--cycle;
	       \draw (e)--(c)--(a)--cycle;
	       \draw (e)--(d);
	       \draw (e)--(f)--(g);
	       \end{tikzpicture}
	    \end{tabular}  &
         8 & \begin{tabular}{c}
	       \begin{tikzpicture}[scale=0.5]
	       \coordinate (a) at (1,-1);
	       \coordinate (b) at (-1,-1);
	       \coordinate (c) at (-.5,-.5);
	       \coordinate (d) at (-.5,.5);
	       \coordinate (e) at (.5,.5);
	       \coordinate (f) at (.5,-.5);
	       \coordinate (g) at (1,1);
	       \coordinate (h) at (-1,1);
	       \draw (b)--(h)--(g)--(a)--cycle;
	       \draw (e)--(d)--(c)--(f)--cycle;
	       \draw (d)--(b)--(c);
	       \draw (h)--(e)--(g);
	       \draw (a)--(f);
	       \end{tikzpicture}
	    \end{tabular} \\
13 & 7 & 7 & \begin{tabular}{c}
			\begin{tikzpicture}[scale=0.5]

\coordinate (a137) at (-1,-1);
\coordinate (b137) at (-1,1);
\coordinate (c137) at (0,.5);
\coordinate (d137) at (210:.5);
\coordinate (e137) at (330:.5);
\coordinate (f137) at (1,-1);
\coordinate (g137) at (1,1);

\draw (a137) -- (b137);
\draw (a137) -- (e137);
\draw (a137) -- (f137);

\draw (b137) -- (c137);
\draw (b137) -- (d137);
\draw (b137) -- (g137);

\draw (c137) -- (d137);
\draw (c137) -- (e137);
\draw (c137) -- (g137);

\draw (d137) -- (e137);

\draw (e137) -- (f137);
\draw (e137) -- (g137);

\draw (f137) -- (g137);

\end{tikzpicture}
	     \end{tabular} 
&
		 8 & \begin{tabular}{c}
	       \begin{tikzpicture}[scale=0.5]
	       \coordinate (a137Dual) at (-.5,.5);
	       \coordinate (b137Dual) at (.5,-.5);
	       \coordinate (c137Dual) at (-.5,-.5);
	       \coordinate (d137Dual) at (-1,1);
	       \coordinate (e137Dual) at (1,1);
	       \coordinate (f137Dual) at (.5,.5);
	       \coordinate (g137Dual) at (1,-1);
	       \coordinate (h137Dual) at (-1,-1);
	       \draw (d137Dual)--(h137Dual)--(g137Dual)--(e137Dual)--cycle;
	       \draw (a137Dual)--(c137Dual)--(b137Dual)--(f137Dual)--cycle;
	       \draw (d137Dual)--(c137Dual)--(h137Dual);
	       \draw (a137Dual)--(b137Dual)--(g137Dual);
	       \draw (e137Dual)--(f137Dual);
	       \end{tikzpicture}
		     \end{tabular}\end{tabular}}
 
 \begin{table}[ht]\centering
   \thirteenedge
 \caption{14 13-edge $\fm$-graphs.\label{tbl:13}}
\end{table}

 \begin{table}[ht]\centering
 \begin{tabular}{|c|c|c|c || c|c|}
      $\#E$ & \# & $\#V$ & $G=(V,E,F)$ &  $\#F$ & $G^*$\\
\hline
14 & 1 & 7   &\begin{tabular}{c}\\
	       \begin{tikzpicture}[scale=0.5]
	       \coordinate (a142V7) at (30:.3);
	       \coordinate (b142V7) at (0:1);
	       \coordinate (c142V7) at (270:.3);
	       \coordinate (d142V7) at (150:.3);
	       \coordinate (e142V7) at (90:1);
	       \coordinate (f142V7) at (180:1);
	       \coordinate (g142V7) at (270:1);

	       \draw (a142V7)--(c142V7)--(d142V7)--cycle;
	       \draw (g142V7)--(b142V7)--(e142V7)--(f142V7)--cycle;
	       \draw (d142V7)--(e142V7)--(a142V7);
	       \draw (a142V7)--(b142V7)--(c142V7);
	       \draw (g142V7)--(c142V7)--(f142V7)--(d142V7);
	       \end{tikzpicture}
	    \end{tabular} &
         9   &\begin{tabular}{c}\\
	      \begin{tikzpicture}[scale=0.5]
	       \coordinate (a142V9D) at (18:.5);
	       \coordinate (b142V9D) at (0:1);
	       \coordinate (c142V9D) at (270:1);
	       \coordinate (d142V9D) at (306:.5);
	       \coordinate (e142V9D) at (234:.5);
	       \coordinate (f142V9D) at (162:.5);
	       \coordinate (g142V9D) at (180:1);
	       \coordinate (h142V9D) at (90:1);
	       \coordinate (i142V9D) at (90:.5);

       \draw (a142V9D)--(d142V9D)--(e142V9D)--(f142V9D)--(i142V9D)--cycle;
       \draw (h142V9D)--(b142V9D)--(c142V9D)--(g142V9D)--cycle;
       \draw (h142V9D)--(i142V9D);
	       \draw (b142V9D)--(a142V9D);
	       \draw (c142V9D)--(d142V9D);
	       \draw (c142V9D)--(e142V9D);
	       \draw (g142V9D)--(f142V9D);
	      \end{tikzpicture}
		   \end{tabular}   \\
		14 & 2 & 7   & \begin{tabular}{c}
				
		\begin{tikzpicture}
		\coordinate (a149V7) at (.5,.5);
		\coordinate (b149V7) at (.25,.25);
		\coordinate (c149V7) at (0,1);
		\coordinate (d149V7) at (0,0);
		\coordinate (e149V7) at (1,0);
		\coordinate (f149V7) at (1,1);
		\coordinate (g149V7) at (.75,.75);
		\draw (c149V7)--(d149V7)--(e149V7)--(f149V7)--cycle;
		\draw (d149V7)--(b149V7)--(a149V7)--(g149V7)--(f149V7);
		\draw (c149V7)--(b149V7)--(e149V7)--(a149V7);
		\draw (e149V7)--(g149V7);

		\draw (.5,.5)--(0,1)--(.75,.75);
		\end{tikzpicture}
	\end{tabular} &
	9   & \begin{tabular}{c}
	       \begin{tikzpicture}[scale=0.5]
		 \coordinate (a149V9Dual) at (.5,-.5);
		 \coordinate (b149V9Dual) at (0,0);
		 \coordinate (c149V9Dual) at (-.5,.5);
		 \coordinate (d149V9Dual) at (.5,.5);
		 \coordinate (e149V9Dual) at (1,1);
		 \coordinate (f149V9Dual) at (-1,1);
		 \coordinate (g149V9Dual) at (-1,-1);
		 \coordinate (h149V9Dual) at (-.5,-.5);
		 \coordinate (i149V9Dual) at (1,-1);

		 \draw (a149V9Dual)--(b149V9Dual)--(h149V9Dual)--cycle;
		 \draw (d149V9Dual)--(b149V9Dual)--(c149V9Dual)--cycle;
		 \draw (f149V9Dual)--(g149V9Dual)--(i149V9Dual)--(e149V9Dual)--cycle;
		 \draw (c149V9Dual)--(f149V9Dual);
		 \draw (d149V9Dual)--(e149V9Dual);
		 \draw (h149V9Dual)--(g149V9Dual);
		 \draw (a149V9Dual)--(i149V9Dual);
	       \end{tikzpicture}
	      \end{tabular} 
\\
14 & 3 & 8   &\begin{tabular}{c}
	      \begin{tikzpicture}[scale=0.5]
	       \coordinate (e141V8) at (-1,-1);
	       \coordinate (h141V8) at (1,-1);
	       \coordinate (a141V8) at (.5,-.5);
	       \coordinate (b141V8) at (.5,.5);
	       \coordinate (g141V8) at (-.5,.5);
	       \coordinate (f) at (-.5,-.5);
	       \coordinate (d141V8) at (-1,1);
	       \coordinate (c141V8) at (1,1);
	       \draw (h141V8)--(c141V8)--(d141V8)--(e141V8)--cycle;
	       \draw (g141V8)--(b141V8)--(a141V8)--(f)--cycle;
	       \draw (b141V8)--(h141V8);
	       \draw (c141V8)--(g141V8)--(d141V8);
	       \draw (e141V8)--(f);
	       \draw (g141V8)--(a141V8);
	       \draw (b141V8)-- (c141V8);
	      \end{tikzpicture}
\end{tabular} 
\\
14 & 4 & 8  & \begin{tabular}{c}
		\begin{tikzpicture}[scale=0.5]

\coordinate (a1410) at (-.5,.5);
\coordinate (b1410) at (.5,-.5);
\coordinate (c1410) at (.5,.5);
\coordinate (d1410) at (1,1);
\coordinate (e1410) at (-1,1);
\coordinate (f1410) at (-.5,-.5);
\coordinate (g1410) at (-1,-1);
\coordinate (h1410) at (1,-1);

\draw (a1410) -- (b1410);
\draw (a1410) -- (c1410);
\draw (a1410) -- (e1410);
\draw (a1410) -- (f1410);

\draw (b1410) -- (c1410);
\draw (b1410) -- (f1410);
\draw (b1410) -- (h1410);

\draw (c1410) -- (d1410);

\draw (d1410) -- (e1410);
\draw (d1410) -- (h1410);

\draw (e1410) -- (f1410);
\draw (e1410) -- (g1410);

\draw (f1410) -- (g1410);

\draw (g1410) -- (h1410);

%
%
		\end{tikzpicture}\end{tabular}
 \\
		14 & 5 & 8   & \begin{tabular}{c}
\begin{tikzpicture}[scale=.5]


\coordinate (a144) at (0,.5);
\coordinate (b144) at (-1,1);
\coordinate (c144) at (1,1);
\coordinate (d144) at (1,-1);
\coordinate (e144) at (0,-.5);
\coordinate (f144) at (0,0);
\coordinate (g144) at (-.5,0);
\coordinate (h144) at (-1,-1);

\draw (a144) -- (b144);
\draw (a144) -- (c144);
\draw (a144) -- (f144);

\draw (b144) -- (c144);
\draw (b144) -- (g144);
\draw (b144) -- (h144);

\draw (c144) -- (d144);
\draw (c144) -- (f144);

\draw (d144) -- (e144);
\draw (d144) -- (h144);

\draw (e144) -- (f144);
\draw (e144) -- (h144);

\draw (f144) -- (g144);

\draw (g144) -- (h144);
\end{tikzpicture}				
%
			       \end{tabular}  \\
14 & 6 & 8   & \begin{tabular}{c}
	       \begin{tikzpicture}[scale=0.5]
	       \coordinate (a) at (.5,.5);
	       \coordinate (b) at (-.5,-.5);
	       \coordinate (c) at (.5,-.5);
	       \coordinate (d) at (1,1);
	       \coordinate (e) at (-1,1);
	       \coordinate (f) at (-.5,.5);
	       \coordinate (g) at (-1,-1);
	       \coordinate (h) at (1,-1);
	       \draw (d)--(h)--(g)--(e)--cycle;
	       \draw (a)--(c)--(b)--(f)--cycle;
	       \draw (d)--(c)--(h);
	       \draw (a)--(b)--(g);
	       \draw (a)--(e)--(f);
	       \end{tikzpicture}
	    \end{tabular}  \\
		14 & 7 & 8   & \begin{tabular}{c}
		 \begin{tikzpicture}[scale=0.8]
		 \coordinate (h) at (90:0.9);
		 \coordinate (b) at (210:1);
		 \coordinate (c) at (330:1);
		 \coordinate (g) at (0,.4);
		 \coordinate (f) at (-.2,0);
		 \coordinate (e) at (.2, 0);
		 \coordinate (a) at (-.3,-.3);
		 \coordinate (d) at (.3,-.3);

		 \draw (h)--(b)--(c)--cycle;
		 \draw (g)--(f)--(e)--cycle;
		 \draw (b)--(f)--(a)--cycle;
		 \draw (c)--(e)--(d)--cycle;
		 \draw (a)--(d);
		 \draw (g)--(h);
		 \end{tikzpicture}\\
		 (Gyrobifastigium ($J_{26}$))
			       \end{tabular} &
		 %
         8   &  
\begin{tabular}{c}
	       \begin{tikzpicture}[scale=0.8]

\coordinate (a147Dual) at (330:1);
\coordinate (b147Dual) at (270:.3);
\coordinate (c147Dual) at (210:1);
\coordinate (d147Dual) at (180:.3);
\coordinate (e147Dual) at (90:.3);
\coordinate (f147Dual) at (0:.3);
\coordinate (g147Dual) at (0:0);
\coordinate (h147Dual) at (90:1);

\draw (a147Dual) -- (b147Dual);
\draw (a147Dual) -- (c147Dual);
\draw (a147Dual) -- (f147Dual);
\draw (a147Dual) -- (h147Dual);

\draw (b147Dual) -- (c147Dual);
\draw (b147Dual) -- (g147Dual);

\draw (c147Dual) -- (d147Dual);
\draw (c147Dual) -- (h147Dual);

\draw (d147Dual) -- (e147Dual);
\draw (d147Dual) -- (g147Dual);

\draw (e147Dual) -- (f147Dual);
\draw (e147Dual) -- (g147Dual);
\draw (e147Dual) -- (h147Dual);

\draw (f147Dual) -- (g147Dual);

		 \end{tikzpicture}
	    \end{tabular} \\
  &&&&&\\
14 & 8 & 8   & \begin{tabular}{c}
	      \begin{tikzpicture}[scale=0.5]
	       \coordinate (a143V8) at (1,-1);
	       \coordinate (b143V8) at (1,1);
	       \coordinate (c143V8) at (.5,-.5);
	       \coordinate (d143V8) at (.5,.5);
	       \coordinate (e143V8) at (-.5,-.5);
	       \coordinate (f143V8) at (-1,-1);
	       \coordinate (g143V8) at (-1,1);
	       \coordinate (h143V8) at (-.5,.5);

	       \draw (b143V8)--(a143V8)--(f143V8)--(g143V8)--cycle;
	       \draw (d143V8)--(c143V8)--(e143V8)--(h143V8)--cycle;
	       \draw (b143V8)--(c143V8)--(a143V8);
	       \draw (d143V8)--(e143V8)--(f143V8);
	       \draw (g143V8)--(h143V8);
	       \draw (h143V8)--(b143V8);
	      \end{tikzpicture}
	    \end{tabular} &
         8   & \begin{tabular}{c}
	      \begin{tikzpicture}[scale=0.5]
	       \coordinate (a143V8D) at (-1,0);
	       \coordinate (b143V8D) at (0,-1);
	       \coordinate (c143V8D) at (1,0);
	       \coordinate (d143V8D) at (-45:.4);
	       \coordinate (e143V8D) at (45:.4);
	       \coordinate (f143V8D) at (0,1);
	       \coordinate (g143V8D) at (135:.4);
	       \coordinate (h143V8D) at (225:.4);

	       \draw (b143V8D)--(a143V8D)--(f143V8D)--(c143V8D)--cycle;
	       \draw (g143V8D)--(e143V8D)--(d143V8D)--(h143V8D)--cycle;
	       \draw (g143V8D)--(f143V8D)--(e143V8D);
	       \draw (d143V8D)--(b143V8D)--(h143V8D);
	       \draw (a143V8D)--(h143V8D);
	       \draw (c143V8D)--(d143V8D);
	      \end{tikzpicture}
	    \end{tabular}  \\
14 & 9 & 8   &\begin{tabular}{c}
	      \begin{tikzpicture}[scale=0.5]
	       \coordinate (a145V8) at (1,-1);
	       \coordinate (f145V8) at (1,1);
	       \coordinate (d145V8) at (.5,-.5);
	       \coordinate (e145V8) at (.5,.5);
	       \coordinate (c145V8) at (-.5,-.5);
	       \coordinate (b145V8) at (-1,-1);
	       \coordinate (h145V8) at (-1,1);
	       \coordinate (g145V8) at (-.5,.5);

	       \draw (f145V8)--(a145V8)--(b145V8)--(h145V8)--cycle;
	       \draw (e145V8)--(d145V8)--(c145V8)--(g145V8)--cycle;
	       \draw (f145V8)--(d145V8)--(a145V8);
	       \draw (c145V8)--(b145V8);
	       \draw (g145V8)--(h145V8);
	       \draw (g145V8)--(f145V8)--(e145V8);
	      \end{tikzpicture}
	    \end{tabular}  &
         8   & \begin{tabular}{c}
		\begin{tikzpicture}[scale=0.5]
		 \foreach \x in {1,...,5}{
		 \draw[rotate=72*\x] (90:1)
		 --(162:1);};
		 \coordinate (a) at (130:.5);
		 \coordinate (f) at (50:.5);
		 \draw (a) -- (162:1);
		 \draw (f) -- (18:1);
		\draw (90:1) -- (a)--(0,-.2)--(f)--cycle;
		 \draw (306:1)-- (0,-.2) -- (234:1);
		 \draw (a) -- (f);
		\end{tikzpicture}
	     \end{tabular}   \\
14 & 10 & 8   & \begin{tabular}{c}
		 \begin{tikzpicture}[scale=0.5]
		 \coordinate (a148V8) at (-1,1);
		 \coordinate (b148V8) at (0,.5);
		 \coordinate (c148V8) at (-.5,0);
		 \coordinate (d148V8) at (.5,0);
		 \coordinate (e148V8) at (0,-.5);
		 \coordinate (f148V8) at (1,-1);
		 \coordinate (g148V8) at (1,1);
		 \coordinate (h148V8) at (-1,-1);

		 \draw (a148V8)--(g148V8)--(f148V8)--(h148V8)--cycle;
		 \draw (b148V8)--(d148V8)--(e148V8)--(c148V8)--cycle;
		 \draw (c148V8)--(d148V8);
		 \draw (h148V8)--(e148V8)--(f148V8);
		 \draw (a148V8)--(b148V8)--(g148V8)--(d148V8);
		 \end{tikzpicture}
	       \end{tabular}%
  
		      \end{tabular}
		     \caption{16 14-edge $\fm$-graphs.\label{tbl:14}}
 \end{table}

\def\inc#1{\begin{tabular}{c}
	    \includegraphics[scale=.08]{Zhang/#1}
	   \end{tabular}}
\begin{table}[ht]\centering
%
 \caption{16 15-edge $\fm$-graphs.\label{tbl:15}}
\end{table}

\begin{table}[ht]\centering
     %
 \caption{12 16-edge $\fm$-graphs.\label{tbl:16}}
\end{table}
\begin{table}[ht]\centering
     %
 \caption{8 17-edge $\fm$-graphs.\label{tbl:17}}
\end{table}
\begin{table}[ht]\centering
     %
 \caption{6 19-edge $\fm$-graphs, and 3 20-edge $\fm$-graphs.\label{tbl:1920}}
\end{table}
\begin{table}[ht]\centering
     %
     \end{tabular}
 \caption{2 21-edge $\fm$-graphs, 1 22-edge $\fm$-graphs, and 3 24-edge $\fm$-graphs.\label{tbl:212224}}
\end{table}

The 116 $\fm$-graphs are presented in Table~\ref{tbl:6891011}, Table~\ref{tbl:12},
Table~\ref{tbl:13}, Table~\ref{tbl:14}, Table~\ref{tbl:15},
Table~\ref{tbl:16}, Table~\ref{tbl:17}, Table~\ref{tbl:18},
Table~\ref{tbl:1920}, and Table~\ref{tbl:212224}. 
Graphs $G=(V,E,F)$ in each row of each table are ordered by $\#V$.  The
rows in each table are ordered by $\#E$, $\#V$, the
vertex-connectedness, and the self-duality of the first graph
$G=(V,E,F)$ of the respective row.  The second column of each table is
for numbering rows among the same $\#E$. For self-dual $G$, we do not present the
dual $G^*$ in the table. Rows $\fm$-E$10$-1, $\fm$-E$12$-8, and
$\fm$-E$13$-1 are indicated with $\dag$. In the three rows, each graph
is not 3-connected.  But in the other rows, each graph is 3-connected.
In the tables of $\fm$-graphs, a graph of $i$ edges, No.~$j$, $k$
vertices is referred by $\fm$-E$i$-$j$V$k$, and the dual graph by
$\fm$-E$i$-$j$V$k$Dual.  

By the classification of $\fm,$ using Theorem~\ref{thm:rp2}, we can determine the class $\fmp^{\R P^2}$ as follows: for any $G\in \fm$ we check whether there exists a fixed-point free, involutive isomorphism of $G$ and the according projection in $\R P^2$ is a tessellation.
We show that $\fmp^{\R
P^2}=\emptyset$.
This completes the proof of Theorem~\ref{thm:classPF}. 
\end{proof}

\section{Graphs on surfaces with positive corner curvature}\label{sec:corner}
\label{sec:PCnC}
In this section, we prove Theorem~\ref{thm:corner curvature}.

\begin{lemma}\label{lem:pcnc}
For any $G=(V,E,F)\in\PCnC$, we have the following:
 \begin{enumerate}
  \item \label{vertex type:pcnc} Both of the vertex degrees and the facial degrees are 3,4,or 5, and
    the vertex pattern of any $d$-valent vertex is $(3^d)$ for any $d=4,5$.

  \item \label{dual:pcnc}  $G^*\in \PCnC$.
 \end{enumerate}
\end{lemma}
By Lemma~\ref{lem:pcnc}~\eqref{vertex type:pcnc},
the minimum combinatorial curvature of the vertex pattern is 1/10.
So the number of vertices is at most 20. By
Lemma~\ref{lem:pcnc}~\eqref{dual:pcnc}, the number of faces is so. It also implies that if we get all the graphs in $\PCnC$ whose $\#V\le\#F$, after dual operation we can get all the graphs in $\PCnC$.
\begin{theorem}
  \label{assert:12} For any $G=(V,E,F)\in\PCnC$, $\#V\le\#F$ implies $\#V \le 12$.
\end{theorem}
 \begin{proof} By the Euler formula, $\#V \le  \#E/2 +1$. Because
  $G$ is a tessellation, $G$ has a medial graph $G'=(V',E',F')$
  and $\#E=\#V'$, by Proposition~\ref{prop:medtess}. By Lemma~\ref{lem:pcnc}~\eqref{vertex type:pcnc},
    the vertex pattern of $G'$ is $(3,3,k,l)$ $(k,l \in \{3,4,5\})$. The minimum combinatorial curvature of $G'$ is $1/15$ given by $(3,3,5,5)$.
    Then $\#V' \le  2/(1/15) = 30$. Hence, $\#V \le  \#E/2 + 1 = \#V'/2 + 1 = 16$.

 Then we need to prove when $13\le\#V\le16$, $G$ is not in $\PCnC$. The number of vertices of degree $d$, the number of faces of facial
degree $d$, the number of vertices, the number of edges, the
number of faces, and the number of vertices of  pattern $p$ are denoted
  by $v_d,f_d,v,e$, $f$, and $n_p$ respectively.
   By simple computation of a  brute-force computer program, the only combinations of nonnegative integers
 $v_3$, $f_3$~($3\le d\le 5$), $v$, $e$, $f$, $n_p$~($p=333$, $334$, $335$,
  $344$, $345$, $355$, $444$, $445$, $455$, $555$)
 such that
 \begin{align*}
     v &=\sum_{d=3}^5 v_d,\qquad   f = \sum_{d=3}^5 f_d,\qquad    v-e+f = 2, \qquad
   2e = \sum_{d=3}^5 d v_d,\\
  v &\le f,\\
 13\le v & \le 16,\\
   v_3 &= n_{333} +n_{334} +n_{335} +n_{344} +n_{345} +n_{355} +n_{444} +n_{445} +n_{455} +n_{555} ,\\
   3f_3 &= 3n_{333} +2n_{334} +2n_{335} +n_{344} +n_{345} +n_{355} +4v_4+5v_5,\\
   4f_4 &= n_{334} +2n_{344} +n_{345} +3n_{444} +2n_{445} +n_{455} ,\\
   5f_5 &= n_{335} +n_{345} +2n_{355} +n_{445} +2n_{455} +3n_{555},\end{align*}
are the following 15 combinations:
 \begin{verbatim}
(1) v3=5,v4=1,v5=7,f3=13,f4=0,f5=3,v=13,e=27,f=16,n333=0,n334=0,
	n335=0,	n344=0,n345=0,n355=0,n444=0,n445=0,n455=0,n555=5

(2) v3=7,v4=1,v5=5,f3=10,f4=0,f5=4,v=13,e=25,f=14,n333=0,n334=0,
	n335=0,n344=0,n345=0,n355=1,n444=0,n445=0,n455=0,n555=6

(3) v3=8,v4=1,v5=4,f3=8,f4=1,f5=4,v=13,e=24,f=13,n333=0,n334=0,
	n335=0,n344=0,n345=0,n355=0,n444=0,n445=0,n455=4,n555=4

(4) v3=8,v4=1,v5=4,f3=8,f4=1,f5=4,v=13,e=24,f=13,n333=0,n334=0,
	n335=0,n344=0,n345=0,n355=0,n444=0,n445=1,n455=2,n555=5

(5) v3=8,v4=1,v5=4,f3=8,f4=1,f5=4,v=13,e=24,f=13,n333=0,n334=0,
	n335=0,n344=0,n345=0,n355=0,n444=0,n445=2,n455=0,n555=6

(6) v3=8,v4=1,v5=4,f3=8,f4=1,f5=4,v=13,e=24,f=13,n333=0,n334=0,
	n335=0,n344=0,n345=0,n355=0,n444=1,n445=0,n455=1,n555=6

(7) v3=5,v4=0,v5=9,f3=15,f4=0,f5=3,v=14,e=30,f=18,n333=0,n334=0,
	n335=0,n344=0,n345=0,n355=0,n444=0,n445=0,n455=0,n555=5

(8) v3=7,v4=0,v5=7,f3=12,f4=0,f5=4,v=14,e=28,f=16,n333=0,n334=0,
	n335=0,n344=0,n345=0,n355=1,n444=0,n445=0,n455=0,n555=6

(9) v3=8,v4=0,v5=6,f3=10,f4=1,f5=4,v=14,e=27,f=15,n333=0,n334=0,
	n335=0,n344=0,n345=0,n355=0,n444=0,n445=0,n455=4,n555=4

(10) v3=8,v4=0,v5=6,f3=10,f4=1,f5=4,v=14,e=27,f=15,n333=0,
	n334=0,n335=0,n344=0,n345=0,n355=0,n444=0,n445=1,n455=2,n555=5

(11) v3=8,v4=0,v5=6,f3=10,f4=1,f5=4,v=14,e=27,f=15,n333=0,
	n334=0,n335=0,n344=0,n345=0,n355=0,n444=0,n445=2,n455=0,n555=6

(12) v3=8,v4=0,v5=6,f3=10,f4=1,f5=4,v=14,e=27,f=15,n333=0,
	n334=0,n335=0,n344=0,n345=0,n355=0,n444=1,n445=0,n455=1,n555=6

(13) v3=9,v4=0,v5=5,f3=9,f4=0,f5=5,v=14,e=26,f=14,n333=0,n334=0,
	n335=0,n344=0,n345=0,n355=2,n444=0,n445=0,n455=0,n555=7

(14) v3=9,v4=0,v5=5,f3=9,f4=0,f5=5,v=14,e=26,f=14,n333=0,n334=0,
	n335=1,n344=0,n345=0,n355=0,n444=0,n445=0,n455=0,n555=8

(15) v3=10,v4=0,v5=6,f3=10,f4=0,f5=6,v=16,e=30,f=16,n333=0,
	n334=0,n335=0,n344=0,n345=0,n355=0,n444=0,n445=0,n455=0,n555=10
 \end{verbatim}

  The last combination is rejected by the following
  Fact~\ref{fact:pcnc}~\eqref{n335:pcnc}, and the other combinations by 
  Fact~\ref{fact:pcnc}~\eqref{deg5:pcnc}.

  \begin{fact}\label{fact:pcnc}
 \begin{align}
  \label{deg5:pcnc}  & v_5=0,\ f_5=0,\ \mbox{or}\ n_{335} >0.\\
  \label{n335:pcnc}  &  n_{335} \ne 1,\ n_{345} >0,\ \mbox{or}\ n_{355} >0.
   \end{align}
  \end{fact}
\begin{proof}

 \eqref{deg5:pcnc} Assume $v_5>0$ and $f_5>0$. Then, there are a 5-valent vertex $x$ and a 5-gon $f$.  
By Lemma~\ref{lem:pcnc}~\eqref{vertex type:pcnc}, for $d=4,5$, any $d$-valent vertex has type $(3^d)$.
So, $x$ is not a vertex of a 5-gon.  Because of the connectedness, there
is a shortest path $P$ from $x=x_0, x_1, x_2, \ldots, x_n$ such that $x_n$ is a vertex of the 5-gon $f$.
The edge $\{x_0,x_1\}$ is shared by two 3-gons, because $x_0$ has type $(3^5)$ by Lemma~\ref{lem:pcnc}~\eqref{vertex type:pcnc}.
If $n=1$, then the degree $d_1$ of $x_1$ is 3, by Lemma~\ref{lem:pcnc}~\eqref{vertex type:pcnc}. Hence $n_{335}>0$.
If $n>1$, then the degree $d_1$ of $x_1$ is 4 or 5, because $P$ is a shortest path. By Lemma~\ref{lem:pcnc}~\eqref{vertex type:pcnc}, the vertex pattern of $x_1$ is $(3^d)$. Hence, the edge $\{x_1,x_2\}$ is shared by two 3-gons.
By repeating this argument, each $\{x_{i-1},x_i\}$ $(1\le i\le n)$ is shared by two 3-gons. 
By Lemma~\ref{lem:pcnc}~\eqref{vertex type:pcnc}, $x_n$ has vertex type $(3,3,5)$. Hence $n_{335}>0$.


\eqref{n335:pcnc}
Assume $n_{335} =1$ and $n_{345} =0$. By $n_{335} =1$, there is a unique vertex $u$ of type
 $(3,3,5)$. The two vertices adjacent to $u$ in the 5-gon should have
 vertex patterns $(3,x,5)$ and $(3,y,5)$ for some $x\ne3$ and some
 $y\ne3$.
 This establishes Fact~\ref{fact:pcnc}.
\end{proof}
Therefore, for any $G=(V,E,F)\in\PCnC$, if $\#V\le\#F$, then
  $\#V\le12$. This completes the proof of Theorem~\ref{assert:12}.
\end{proof}

\begin{proof}[Proof of Theorem~\ref{thm:corner curvature}]
By Lemma~\ref{lem:pcnc} and Theorem~\ref{assert:12}, we enumerate the finite set $\PCnC$ by a computer program. We first classify planar graphs in $\PCnC.$ The program
is a modification of Brinkmann-McKay's \texttt{plantri}. Their program
\texttt{plantri} can enumerate all the simple, 2-connected, planar
graphs such that all facial degree is greater than 2 and less than 6 and
the number $v$ of vertices is a given number.  We modify
\texttt{plantri} so that any vertex pattern is one of
\begin{align*}
&(3^d)\ (d=3,4,5),\quad (3,3,4), (3,3,5), (3,4,4), (3,4,5), (3,5,5),\\
&(4,4,4), (4,4,5), (4,5,5), (5,5,5),
\end{align*}
and $v$ is at most the number of faces, based on
Lemma~\ref{lem:pcnc}~\eqref{vertex type:pcnc},\eqref{dual:pcnc}.  Then,
we run the modified \texttt{plantri} with the number of vertices
$4,\ldots,12$.  It outputs 13 graphs. By
Lemma~\ref{lem:pcnc}~\eqref{dual:pcnc}, we take the dual graphs into
account. Then we obtain 22 simple, 2-connected, planar graphs such that
corner curvature is positive everywhere.  Table~\ref{tbl:PCnC1} and Table~\ref{tbl:PCnC} present
the 22 graphs. There are one 4-edge self-dual graph, one 8-edge self-dual
graph, and one 10-edge self-dual graph. The other 19 graphs are not self-dual.
All the 22 planar graphs of $\PCnC$ are actually 3-connected. 

From the classification of planar graphs in $\PCnC,$ see
 Table~\ref{tbl:PCnC1} and Table~\ref{tbl:PCnC}, by Theorem~\ref{thm:rp2}, we obtain the classification of graphs on $\R P^2$ in $\PCnC.$ By similar case-by-case checking as in the proof of Theorem~\ref{thm:classPF}, it consists of hemi-icosahedron and hemi-dodecahedron, see e.g. \cite{McMullen92}, which are the
projections of the regular icosahedron and the regular dodecahedron on $\SP^2$ into $\R P^2$ via the double cover $\pi:\SP^2\to \R P^2.$
\end{proof}

\begin{table}[ht]\centering
\begin{tabular}{|c|c|c|c || c|c|c|}
      $\#E$ & \# & $\#V$ & $G=(V,E,F)$ &  $\#F$ & $G^*$ \\
\hline
\hline
4 & 1 & 4  & Regular tetrahedron  &4& \\ 
\hline
8 & 1 & 5  & Square pyramid ($J_1$)&5&   \\ 
\hline
9 & 1 & 5  & \begin{tabular}{l}$3$-gonal\\ bipyramid ($J_{12}$)\end{tabular}&5& $3$-gonal prism  \\ 
\hline
	       10 & 1 & 6  & \begin{tabular}{l}$5$-gonal\\
			      pyramid ($J_2$)\end{tabular} &6&    \\ 
\hline
      12 & 1 & 6  & \begin{tabular}{c}\\
	       \begin{tikzpicture}[scale=.5]
\begin{scope}[scale=0.9]
\foreach \x in {1,...,3}
{
\draw[rotate=120*\x]
  (90:2) -- (210:2);
\draw[rotate=120*\x]
  (90:2) -- (0:0);

 }
\draw (90:2) -- (150:0.4) -- (210:2);
  \draw (0:0) -- (150:0.4);
\begin{scope}[xscale=-1]
 \draw (90:2) -- (150:0.4) -- (210:2);
\draw (0:0) -- (150:0.4);  x
\end{scope}
	       \end{scope}\end{tikzpicture}\\
		     (Biaugmented \\
		    tetrahedron)\end{tabular}
  &
 8  & \begin{tabular}{c}
       \begin{tikzpicture}[scale=.5]
       \begin{scope}[rotate=90,xscale=-1]

\foreach \x in {1,...,5}
{
\draw[rotate=72*\x]
	(0:0.8) -- (72:0.8);
}

 \draw (0:.8) --(0:2) -- (98:2);
 \draw (262:2) -- (0:2);
 \draw (216:.8) --(262:2) -- (288:.8);
 \draw (72:.8) --(98:2) -- (144:.8);
\end{scope}     \end{tikzpicture}      \end{tabular}   \\ 
\hline
12 & 1 & 6  & Regular octahedron & 8  & Cube   \\ 
\hline
15 & 1 & 7  & \begin{tabular}{l}$5$-gonal\\ bipyramid ($J_{13}$)\end{tabular}& 10 & $5$-gonal prism   
\end{tabular}
 \caption{All the 22 2-connected, simple, planar graphs such that every corner
 curvature is positive~(1/2).  The graphs in the first four, the sixth, and the
 seventh lines are in $\fm$, but the two graphs in the fifth line are not.\label{tbl:PCnC1}}
\end{table}

\begin{table}[ht]\centering
\begin{tabular}{|c|c|c|c || c|c|c|}
      $\#E$ & \# & $\#V$ & $G=(V,E,F)$ &  $\#F$ & $G^*$ \\
\hline
\hline
	      15 & 2 & 7  &\begin{tabular}{c}\\
			    \begin{tikzpicture}[yshift=1cm,scale=0.5]
\foreach \x in {1,...,3}
{
\draw[rotate=120*\x]
  (90:2) -- (210:2)--(150:0.5)--cycle;
\draw[rotate=120*\x]
  (0:0) --  (150:0.5) -- (270:0.5);
 }
			    \end{tikzpicture}\\
	      (Augmented octahedron)
     \end{tabular} &
	  10  &
	      \begin{tabular}{c}\\
	       \begin{tikzpicture}[scale=0.5]

\begin{scope}[scale=0.6,xshift=10cm]
\foreach \x in {0,...,2}
{
\draw[rotate=120*\x]  (210:3)--(90:3) -- (90:1) -- (150:1)--(0:0);
 \draw[rotate=120*\x] (210:1)--(150:1);
}
\end{scope}
	       \end{tikzpicture}\end{tabular}\\
\hline
  18 & 1 & 8  & \begin{tabular}{c}\\
			             \begin{tikzpicture}[yshift=1cm,scale=0.5]
\foreach \x in {1,...,3}
{
\draw[rotate=120*\x]
  (90:1.8) -- (210:2)--(150:0.5)--cycle;
\draw[rotate=120*\x]
  (0:0) --  (150:0.5) -- (270:0.5);
					 }
					 \draw (210:2) -- (90:2.5) -- (330:2);
					 \draw (90:2) -- (90:2.5);
					\end{tikzpicture}
	       \\

	       (Gyroelongated\\ triangular\\
	       bipyramid)\\
			     \end{tabular} &
12  & \begin{tabular}{c}
\begin{tikzpicture}[scale=.5]
		\begin{scope}[scale=0.6,xshift=10cm]
\foreach \x in {0,...,2}
{
\draw[rotate=120*\x]  (210:3)--(90:3) -- (90:1) -- (150:1)--(150:.5)--(30:.5);
 \draw[rotate=120*\x] (210:1)--(150:1);
}
\end{scope}
\end{tikzpicture}\\
	(Snub $3$-gonal prism)\\
	\cite[p.20]{MR2429120}
       \end{tabular}
  \\ 
\hline
18 & 2 & 8  & \begin{tabular}{c}\\
\begin{tikzpicture}[scale=0.5]
\begin{scope}[scale=0.9]
\foreach \x in {1,...,3}
{
\draw[rotate=120*\x]
  (90:2) -- (210:2)--(150:0.5)--cycle;
}

 \draw  (0:0) --  (150:0.5) -- (270:0.5)--(90:1);
 \draw  (0:0) --  (30:0.5) -- (270:0.5)--(90:1);

 \draw (30:.5) -- (90:1) -- (150:.5);
 \draw (90:2) -- (90:1);
\end{scope}
\end{tikzpicture} \\
	       (Snub disphenoid~($J_{84}$))
\end{tabular}
&
 12  & \begin{tabular}{c}\begin{tikzpicture}[scale=0.5]
\begin{scope}[scale=0.6,xshift=10cm]
\draw  (135:3)--(135:2) -- (150:1) -- (225:2)--(225:3)--cycle;
 \draw (0:0)--(150:1)--(225:2)--(270:1)--cycle;
 \draw (45:2)--(135:2);
 \draw (45:3)--(135:3);
 \draw (225:3)--(315:3);
\end{scope}
\begin{scope}[scale=0.6,xshift=10cm,xscale=-1]
\draw  (135:3)--(135:2) -- (150:1) -- (225:2)--(225:3)--cycle;
 \draw (0:0)--(150:1)--(225:2)--(270:1)--cycle;
 \draw (45:2)--(135:2);
 \draw (45:3)--(135:3);
 \draw (225:3)--(315:3);
\end{scope}
	       \end{tikzpicture}\end{tabular}   \\ 
\hline
21 & 1 & 9  & \begin{tabular}{c}\\
\begin{tikzpicture}[scale=0.5]
	       \begin{scope}[scale=0.6,yshift=1.5cm]
\foreach \x in {0,...,2}
{
\draw[rotate=120*\x]  (210:3)--(90:3) -- (90:1) -- (150:1)--cycle;
 \draw[rotate=120*\x] (210:1)--(150:1)--(90:3);
  \draw[rotate=120*\x] (-150:1)--(90:1);
}
\end{scope}
	       \end{tikzpicture} \\ (Triaugmented $3$-gonal \\
	      prism ($J_{51}$))\end{tabular}		    &
13  & \begin{tabular}{c}
       \begin{tikzpicture}[xscale=.25,yscale=.2,rotate=30]
	 \foreach \x in {0,...,2}
	 {
	 \draw[rotate=120*\x]
	 (0:0)--(0:1)--(30:2)--(90:2)--(120:1)--cycle;
	 \draw[rotate=120*\x] (30:2)--(0:4)--(-30:2);
	 }
	 \draw (120:4)--(240:6)--(0:4);
	 \draw (240:4)--(240:6);
      \end{tikzpicture}\\
      \end{tabular}   \\ 
\hline
	       24 & 1 & 10 & \begin{tabular}{c}\\
			      	       \begin{tikzpicture}[scale=0.2,xshift=3,yshift=-13]
		\begin{scope}
		 \draw
		 (-1,1)--(1,1)--(0,4)--(4,-3)--(0,-2)--(1,-1)--(0,-1)--(0,0)--(1,-1)--(1,1)--(0,0);
		 \draw (1,1)--(4,-3)--(1,-1);
		\end{scope}
		\begin{scope}[xscale=-1]
		 \draw
		 (-1,1)--(1,1)--(0,4)--(4,-3)--(0,-2)--(1,-1)--(0,-1)--(0,0)--(1,-1)--(1,1)--(0,0);
		 \draw (1,1)--(4,-3)--(1,-1);
		\end{scope}
		\draw (0,-1)--(0,-2);
		 \draw (4,-3)--(-4,-3);
	       \end{tikzpicture}\\
	       (Gyroelongated square\\
	       bipyramid ($J_{17}$))\\
			     \end{tabular}&
 16 &  \begin{tabular}{c}\begin{tikzpicture}[scale=.5]

\begin{scope}[scale=0.6,xshift=10cm,yshift=1.5cm,yscale=-1]
\foreach \x in {0,...,3}
{
\draw[rotate=90*\x]  (180:2)--(90:2) -- (90:1) -- (135:1)--(135:.5)--(45:.5);
 \draw[rotate=90*\x] (180:1)--(135:1);
}
\end{scope}
      \end{tikzpicture}\\ (Snub square prism)\\  \cite[p.20]{MR2429120}\end{tabular}\\  
\hline
30 & 1 & 12 & Regular icosahedron      &
		      20 & \begin{tabular}{l}Regular\\ dodecahedron\end{tabular}      \\  
      \hline
 \end{tabular}
 \caption{All the 22 2-connected, simple, planar graphs such that every corner
 curvature is positive (2/2).  All the graphs in this table are not in $\fm$.\label{tbl:PCnC}}
\end{table}



   \begin{figure}[ht]
    \centering
    \tikzset{>=latex}
\begin{tikzpicture}
\begin{scope}
\foreach \x in {1,...,5}
{
\coordinate (A\x) at (18+72*\x:0.8);
\coordinate (B\x) at (18+72*\x:1.5);
\coordinate (C\x) at (-18+72*\x:2);
 \coordinate (D\x) at (-18+72*\x:3);
}
 \draw[solid,->]  (C5) --  (B5);
 \draw[dashed,->]  (B3) -- (C4);
 \draw[very thick,dotted,->]  (B5) -- (C1);
 \draw[very thick,solid,->]  (B1) --  (C2);
  \draw[very thick,dashed,->] 	(C2) -- (B2);

 \draw[solid,->]   (B2) -- (C3);
 \draw[dashed,->] (C1) -- (B1);
 \draw[very thick,dotted,->] (C3) -- (B3);
 \draw[very thick,solid,->]  (C4) --  (B4);
 \draw[very thick,dashed,->] (B4) -- (C5);
\draw
	(A1) -- (A2) -- (A3)
	(A2) -- (B2)
	(A4) -- (B4);

\draw
	(A1) -- (A5) -- (A4)
	(A5) -- (B5)
	(A3) -- (B3);

\draw
	(A3) -- (A4)
	(A1) -- (B1)
	(B2) -- (C3);
\end{scope}
    \begin{scope}[scale=1,xshift=5cm,yscale=-1]
       \foreach \x in {1,...,5}
       {
 \coordinate (F\x) at (126+72*\x:1.4);
 \coordinate (G\x) at (54+72*\x:1.4);
 \coordinate (H\x) at (90+72*\x:2.2);
       \coordinate (I\x) at (90+72*\x:3.8);
       }
\foreach \x in {1,...,5}
 {
 \draw (0:0) -- (G\x);
 \draw (G\x) -- (F\x);
     }
     \draw[dashed,->] (H3) -- (G3);
     \draw[very thick,->] (G3) -- (H2);
     \draw[very thick,dashed,->] (H2) -- (G2);
     \draw[solid,->] (G2) -- (H1);
     \draw[very thick,dotted,->] (H1) -- (G1);
     \draw[dashed,->] (G1) -- (H5);
     \draw[very thick,->] (H5) -- (G5);
     \draw[very thick,dashed,->] (G5) -- (H4);
     \draw[solid,->] (H4) -- (G4);
     \draw[very thick,dotted,->] (G4) -- (H3);
     \end{scope}
    \end{tikzpicture}
    \caption{The only simple, 2-connected, projective planar graphs with
    positive corner curvature: Hemi-dodecahedron~(left) and hemi-icosahedron~(right).
In each,  arrows with the same style are identified.
    \label{fig:projective
 planar grpahs with positive corner curvature}}
  \end{figure}
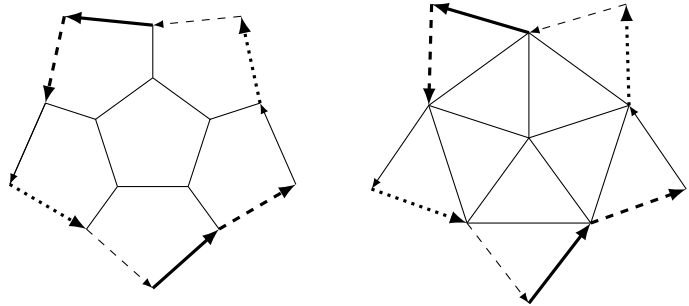


\end{document}